\documentclass{amsart}
\usepackage{latexsym,amssymb,amsmath}
\numberwithin{equation}{section}
\newtheorem{theorem}[equation]{Theorem}
\newtheorem{lemma}[equation]{Lemma}
\newtheorem{corollary}[equation]{Corollary}
\newtheorem{proposition}[equation]{Proposition}
\newtheorem{definition}[equation]{Definition}

\newtheorem{remark}[equation]{Remark}
\newcommand{\Stab}{\mathop{\textrm{Stab}}}
\def\Sym{{\rm Sym}}
\def\GL{{\rm GL}}
\def\End{{\rm End}}
\begin{document}
\openup 1.8\jot
\title[Cyclic matrices and pairwise non-commuting elements]{Abelian 
coverings of finite general linear groups and 
an application to  their non-commuting graphs}   

\author[A. Azad ]{Azizollah Azad }
\address{Azizollah Azad, Department of Mathematics,\newline
Arak University, Arak, 38156-879, Iran.}
\email{a-azad@araku.ac.ir}
\author[M. A. Iranmanesh]{Mohammad A. Iranmanesh}
\address{Mohammad A. Iranmanesh, Department of Mathematics, \newline
  Yazd University,  Yazd, 89195-741 , Iran}
\email{iranmanesh@yazduni.ac.ir}
\author[C. E. Praeger]{Cheryl E. Praeger}
\address{Cheryl E. Praeger,  School of Mathematics and Statistics,\newline
The University of Western Australia,
 Crawley, WA 6009, Australia} \email{praeger@maths.uwa.edu.au}
\author[P. Spiga]{Pablo Spiga}
\address{Pablo Spiga,  School of Mathematics and Statistics,\newline
The University of Western Australia,
 Crawley, WA 6009, Australia} \email{spiga@maths.uwa.edu.au}

\thanks{Address correspondence to A. Azad, Department of Mathematics,
  Arak University, Arak 38156-879, Iran; 
E-mail: a-azad@araku.ac.ir\\ 
The paper forms part of the Australian Research Council Federation Fellowship Project
FF0776186 of the third author. The fourth author is supported by UWA as part of the
Federation Fellowship project. The second author is supported by Research Council of Yazd University.
}


\begin{abstract}
In this paper we introduce and study a family $\mathcal{A}_n(q)$ of
abelian subgroups of $\GL_n(q)$ covering every element of
$\GL_n(q)$. We show that $\mathcal{A}_n(q)$ contains all the centralisers
of cyclic matrices and equality holds if $q>n$. Also, for $q>2$, we
prove a simple closed 
formula for the size of 
$\mathcal{A}_n(q)$ and give an upper bound if $q=2$.

A subset $X$ of a finite group $G$ is said to be pairwise 
non-commut\-ing  
if $xy\not=yx$, for distinct elements $x, y$ in $X$. As an application
of our results on $\mathcal{A}_n(q)$, we prove
lower and upper  bounds for the maximum size of a pairwise
non-commuting subset of $\GL_n(q)$. (This is the clique number of the
non-commuting graph.) Moreover, in the case where $q>n$,
we give an explicit formula for the maximum size of a pairwise
non-commuting  set.
\\ \\
\begin{center}
\emph{For the $100$th anniversary of the birth of B. H. Neumann}
\end{center}
\end{abstract}

\maketitle

\section{Introduction}\label{sec:introd}

In a finite general linear group $\GL_n(q)$ the class of \emph{cyclic
  matrices} (see Section~\ref{sec:1.1} for the definition)
plays an important role both algorithmically
(see~\cite{NPMeataxe}), and in representation theory (for the recognition of
irreducible representations).  This paper uncovers a new role in which
cyclic matrices help to 
determine the maximum size $\omega(\GL_n(q))$ of a set of \emph{pairwise
non-commuting} elements of $\GL_n(q)$,
or equivalently the clique number of the non-commuting graph for this
group. The study of these clique numbers for various families of
groups goes back to the 1976 paper [14] of B. H. Neumann, answering a
question of Paul Erd\"os, and inspired much subsequent research. A
short account is given in Subsection 1.1 to contextualise our
results. Our paper is written in honour of the 100th anniversary of
B. H. Neumann's birth on 15 October 1909.  Our results on
$\omega(\GL_n(q))$ can be summarised as follows in terms of the
quantity
\begin{equation}\label{eq:lq}
l(q)=\prod_{k=1}^\infty(1-q^{-k})^{-k(k+1)/2-1}. 
\end{equation}

\begin{theorem}\label{thm:summarise}
For $q\geq 2$,
$$q^{-n}(1-q^{-3}-q^{-5}+q^{-6}+q^{-n})
\leq\frac{\omega(\GL_n(q))}{|\GL_n(q)|}
\leq q^{-n}l(q)$$
and moreover
$$l(q)\leq
\left\{
\begin{array}{ccl}
1+2q^{-1}+7q^{-2}+114q^{-3}&&\textrm{if }q\geq 3\\
395.0005&&\textrm{if }q=2.
\end{array}
\right.
$$
\end{theorem}

We comment separately below on our strategies for proving the lower
bound and the upper bound. These involve on the one hand existing
estimates for the proportion of cyclic matrices in $\GL_n(q)$ (for the
lower bound), and on the other hand a new investigation of a
 family $\mathcal{A}_n(q)$ of
abelian subgroups of $\GL_n(q)$ covering every group element.

\subsection{Cyclic matrices and the lower bound}\label{sec:1.1}

An element $g\in \GL_n(q)$ is cyclic if its characteristic polynomial
is equal to its minimal polynomial.

\begin{definition}\label{def:centr}
{\rm We denote by $\mathcal{N}_n(q)$ the set of \emph{centralisers of cyclic
matrices} in $\GL_n(q)$, that is,
$\mathcal{N}_n(q)=\{C_{\GL_n(q)}(g)\mid g \textrm{ 
  cyclic in }\GL_n(q)\}$. Also, we denote by $N_n(q)$ the cardinality
$|\mathcal{N}_n(q)|$.}  
\end{definition}

The centralisers of cyclic matrices in $\GL_n(q)$ are
abelian and small (when compared with centralisers of non-cyclic
matrices). Moreover they cover a \emph{large} fraction of the elements
of  $\GL_n(q)$. It turns out that results of
Neumann and 
Praeger~\cite{NP}, and more 
precise results obtained independently by Fulman~\cite{Fu} and
Wall~\cite{Wa}, can be  
applied almost directly to give a very good lower bound on $N_n(q)$
and on $\omega(\GL_n(q))$.  

\begin{theorem}\label{thm:lowerbound}
  $$\frac{\omega(\GL_n(q))}{|\GL_n(q)|}\geq
  \frac{N_n(q)}{|\GL_n(q)|}\geq q^{-n}(1-q^{-3}-q^{-5}+q^{-6}-q^{n}).$$  
\end{theorem} 

\subsection{An abelian covering and the upper bound}\label{sec:1.2}

Centralisers of cyclic matrices do not cover every element in
$\GL_n(q)$ (see Remark~\ref{rm:0} at the end of
Section~\ref{sec:ideaoftheproof}).  
Therefore, in Definition~\ref{def:An}
below we define a family $\mathcal{A}_n(q)$ of abelian subgroups of
$\GL_n(q)$, which contains $\mathcal{N}_n(q)$ and covers
every element of $\GL_n(q)$, giving an upper bound for $\omega(\GL_n(q))$. We 
prove moreover that  $\mathcal{N}_n(q)=\mathcal{A}_n(q)$ 
when $q> n$. 

\begin{theorem}\label{thm:sumupthms}
\begin{description}
\item[$(a)$]$\GL_n(q)=\cup_{A\in \mathcal{A}_n(q)}A$. 
\item[$(b)$]$\omega(\GL_n(q))\leq |\mathcal{A}_n(q)|$ with equality if
  and only if $q>n$.
\item[$(c)$]$\mathcal{N}_n(q)\subseteq \mathcal{A}_n(q)$ with equality
  if and only if $q>n$.  
\end{description}
\end{theorem}

Thus we need to determine $|\mathcal{A}_n(q)|$. It seems, on
consulting several of our colleagues that the family 
$\mathcal{A}_n(q)$ has not been studied previously. This surprised us,
considering the central role it plays in
Theorem~\ref{thm:sumupthms}. The set
$\mathcal{A}_n(q)$ is defined as follows.

\begin{definition}\label{def:An}
{\rm Let $V$ be the $n$-dimensional vector space $k^n$ over the
field $k$ of size $q$. Let
$\mathcal{A}_n(q)$ be the set of abelian subgroups $A$ of $\GL_n(q)$ 
such that the $A$-module $V$ has a decomposition $V_1\oplus\cdots
  \oplus V_r$ into indecomposable $A$-modules satisfying the following
  properties: 
\begin{description}
\item[$(i)$]$A=A_1\times \cdots\times A_r$, where $A_i\subseteq \GL
  (V_i)$;
\item[$(ii)$]for $i=1,\ldots,r$, we have $A_i=C_{\GL
  (V_i)}(a_i)$, for some element $a_i\in \GL (V_i)$ such
  that $V_i$ is an 
  indecomposable $\langle a_i\rangle$-module.
\end{description}
}
\end{definition}

It is shown in Proposition~\ref{prop:numberA_n} that, for $q>2$, the elements in
$\mathcal{A}_n(q)$ are maximal abelian subgroups of $\GL_n(q)$.  The
bulk of the paper is devoted to determining  the limiting size of 
$\mathcal{A}_n(q)$ for a fixed field size $q>2$ and large n. 

\begin{theorem}\label{thm:limiting}For $q>2$, the
  sequence $\{q^n|\mathcal{A}_n(q)|/|\GL_n(q)|\}_{n\geq 0}$ is
  increasing with limit $l(q)$, as defined in Equation~\ref{eq:lq}, and we have
  
$$\left|q^n\frac{|\mathcal{A}_n(q)|}{|\GL_n(q)|}-l(q)\right|=o(r^{-n/2})$$
for every positive $r<q$. For $q=2$, there exists an increasing sequence
$\{2^nb_n\}_{n\geq 0}$ with limit $l(2)$ such that
$|\mathcal{A}_n(2)|/|\GL_n(2)|\leq b_n$, and we have
$$\left|2^nb_n-l(2)\right|=o(r^{-n/2}) $$
for every positive $r<2$. 
\end{theorem}

The sequence $\{b_n\}_{n\geq 1}$ can be found in
Definition~\ref{def:F}.  Our method is to study the generating
function $F(t)=\sum_n^\infty a_n 
t^n$, where $a_n=|\mathcal{A}_n(q)|/|\GL_n(q)|$ and $a_0=1$ (see
Definition~\ref{def:F}). For $q>2$ we obtain a simple formula for
$F(t)$. In fact, using the
theory of symmetric functions
 we prove the 
following result. 
\begin{theorem}\label{thm:F}If $q>2$, then
\begin{equation*}
F(t)=\left(\prod_{i=0}^\infty(1-q^{-(i+1)}t)^{-1}\right)\left(
\prod_{m\geq 2,i,j\geq 0}(1-q^{-(i+j+2m-1)}t^m)^{-1}\right).
\end{equation*}
Moreover, $F(t)$ has a  simple pole at $t=q$ and $(1-q^{-1}t)F(t)$ is
analytic on a disk of radius $q^{3/2}$.  
\end{theorem}

It would be interesting to know if a similar formula could be obtained
for $F(t)$ when $q=2$. 

\begin{remark}\label{rm:e}
The coefficient of degree $n$ in $F(t)$ equals
$|\mathcal{A}_n(q)|/|\GL_n(q)|$ which also equals $\omega(\GL_n(q))/|\GL_n(q)|$
when $q>n$, by Theorem~\ref{thm:sumupthms}~$(b)$. In particular, 
the equation for $F(t)$ in 
Theorem~\ref{thm:F} can be used to obtain explicit formulas  for
$\omega(\GL_n(q))$ when $q>n$. 

We present in
Table~\ref{table:1}   $|\mathcal{A}_n(q)|$,
for $1\leq n\leq 6$ and $q>2$. 

\begin{table}[!h]
\[\begin{array}{|c|l|}\hline
n &|\mathcal{A}_n(q)|\\\hline
1 & 1\\
2 & q^2+q+1\\
3 & q^6 + q^5 + 3q^4 + 3q^3 + q^2 - q -  1\\
4 & q^{12} + q^{11} + 4q^{10} + 7q^9 + 9q^8 + 5q^7 + 2q^6 - 3q^5 -
2q^4 - q^3 + q^2    + q\\
5&q^{20} + q^{19} + 4q^{18} + 9q^{17} + 18q^{16} + 22q^{15} +
22q^{14} + 15q^{13} + 6q^{12} -4q^{11}- 7q^{10}\\
& - 6q^9 - 2q^8 + q^7 +2q^6 + q^5 - q^4 - q^3\\ 
6&q^{30} + q^{29} + 4q^{28} + 10q^{27} + 23q^{26} + 40q^{25} +
60q^{24} + 65q^{23} + 68q^{22}+ 53q^{21}\\
&+ 33q^{20} + 5q^{19} -8q^{18} - 19q^{17} - 16q^{16} - 7q^{15} + q^{14} +  
    6q^{13} + 6q^{12} + 5q^{11}\\
& - q^9 - q^8 + q^7 + q^6\\\hline 
\end{array}\]
\caption{}\label{table:1}
\end{table}
\end{remark}

In the next subsection, we recall how the problem of 
determining the maximum size $\omega(G)$ of a pairwise non-commuting
set of elements arises in group theory. Also, we recall some
results on $\omega(G)$, for various families of groups, and we see how
our method generalises some results in the literature~\cite{AAM,AP}.

\subsection{The non-commuting graph of a group}\label{sec:sub}

In 1976, B. H. Neumann \cite{Neumann} answered a question of Paul Erd\H{o}s
about the maximal \emph{clique} size in the \emph{non-commuting graph}
$\Gamma(G)$ of a group $G$, namely  the graph with 
vertices the elements of $G$ and with edges the pairs $\{x,y\}$ with
$xy\neq yx$. A clique in a graph is a set of pairwise adjacent
vertices and hence in $\Gamma(G)$ a clique is a \emph{pairwise non-commuting
subset of $G$}. In group theoretic language, Erd\H{o}s asked whether
there exists a finite 
upper bound on the cardinalities of pairwise non-commuting
subsets of $G$,
assuming that every  
such subset is finite. Neumann proved that the family of groups satisfying
the condition of Erd\H{o}s is precisely the class of groups $G$ in which
the centre $Z(G)$ has finite index, and he proved moreover that for
such groups $G$, each pairwise non-commuting subset of $G$ has size at
most $|G:Z(G)|-1$. Neumann's answer  
inspired much subsequent research, for example
\cite{AAM,AP,Isaacs,Brown,Chin,Mckenzie,Pyber,T}.   

We let $\omega(G)$ denote the \emph{maximum cardinality} of a pairwise
non-commuting subset of $G$. Because of Neumann's result, the study of
groups $G$ such that 
$\omega(G)<\infty$ is reduced to the study of finite groups.
It follows from results of \cite{Pyber} that, if
$n=|G:Z(G)|$, then $c\, {\log_2}n \leq \omega 
(G) \leq n-1$, for some positive constant $c$. The lower bound is achieved with
$c=1$ by each extraspecial $2$-group. (According to
\cite{Isaacs,Pyber}, this was proved by Isaacs.)  

On the other hand the upper bound is achieved for the quaternion group
$G=Q_8$ and the dihedral group $G=D_8$ of order 8, both of which have
$\omega(G)=3$. It is believed that groups $G$ which are ``close'' to
being nonabelian simple will have $\omega(G)$ ``close'' to the upper
bound. Indeed, for the symmetric group $\Sym(n)$ of degree $n$
$\omega(\Sym(n))$ satisfies the bounds in Table~\ref{table:2}, where
$a,b,c,d$ are  constants, see~\cite[Theorem~1]{Brown}.

\begin{table}[!h]
\[\begin{array}{|lll|}\hline
\textrm{Lower bound} &&\textrm{comments}\\\hline 
a(n-2)!              &&\textrm{for  every }n\\
 d(\log\log n)(n-2)! &&\textrm{for infinitely many } n\\\hline
\textrm{Upper bound} &&\\\hline
b(\log\log n)(n-2)!  &&\textrm{for every }n\\
c(n-2)!              &&\textrm{for infinitely many }n\\\hline
\end{array}\]
\caption{Results for $\omega(\Sym(n))$ from Brown~\cite{Brown}}\label{table:2}
\end{table}
Also, for finite general linear groups $\GL_n(q)$, if $q>2$ then $\omega 
(\GL_2(q))=q^2+q+1$ (see~\cite[Lemma 4.4]{AAM}), and if $q>3$ then 
 $\omega (\GL_3(q))=q^6+q^5+3q^4+3q^3+q^2-q-1$ (see~\cite[Theorem
  1.1]{AP}). These two results on the finite general linear group
prove the bounds of  
Theorem~\ref{thm:summarise} for $n\leq 3$ and the values for
$\omega(\GL_n(q))$ (for $n=2,3$) are exactly the second and third
row in Table~\ref{table:1}.

Finally, it was conjectured~\cite[Conjecture  
  1.2]{AP} that, for $q>n$, the number 
$\omega (\GL_n(q))$ should be somewhat larger than
$q^{n^2-n}+q^{n^2-n-1}+(n-1)q^{n^2-n-2}$. As we discuss in
Remark~\ref{rm:2}, this conjecture is incorrect for $n\geq 6$.

Using the results of~\cite{FNP} a lower bound for $\omega(G)$ similar to
that provided by Theorem~\ref{thm:summarise} can be obtained for all finite
classical groups $G$. It would be interesting to know if a similar
estimation for a class of abelian subgroups of classical groups could
be carried out to yield a good upper bound for $\omega(G)$ for these groups.

\subsection{Structure of the paper}\label{sec:1.3}
In this
final introductory section we briefly summarise where the proofs of
the theorems stated in Section~\ref{sec:introd} are given.\\

\noindent\emph{Proof of Theorem~\ref{thm:summarise}. }The lower bound
is a direct application of Theorem~\ref{thm:lowerbound}. 
From Theorem~\ref{thm:sumupthms} we have $\omega(\GL_n(q))\leq
|\mathcal{A}_n(q)|$ and from Theorem~\ref{thm:limiting} the
sequence $\{q^n|\mathcal{A}_n|/|\GL_n(q)|\}_n$ is increasing with
limit $l(q)$. Therefore the upper bound follows. 
The estimates on $l(q)$ are collected in
Lemma~\ref{lemma:estimates} in Section~\ref{sec:ana}.\qed\\

\noindent\emph{Proof of Theorem~\ref{thm:lowerbound}. } The proof of
this result is given
in Section~\ref{sec:LowerBound}.\qed\\

\noindent\emph{Proof of Theorem~\ref{thm:sumupthms}. }This theorem is
proved in Section~\ref{sec:thefamily}. Namely, Part~$(a)$ is proved in
Proposition~\ref{prop:coveringproperty}, Part~$(b)$
(which is an application of Part~$(a)$) is proved in
Corollary~\ref{cor:exact} and Part~$(c)$ is proved in
Theorem~\ref{thm:exact}.\qed\\

\noindent\emph{Proof of Theorem~\ref{thm:limiting}. } This result in
proved in Sections~\ref{sec:F} and~\ref{sec:ana}. Namely, in
Theorem~\ref{thm:monotonic} we prove that the sequences
$\{q^n|\mathcal{A}_n(q)|/|\GL_n(q)|\}_n$ (for $q>2$) and 
$\{q^nb_n\}_{n}$ (for $q=2$) are increasing. In
Theorem~\ref{thm:limit} we compute the rate of convergence and the
limit.\qed\\

\noindent\emph{Proof of Theorem~\ref{thm:F}. }The equation for the
generating function $F(t)$ is proved in 
Theorem~\ref{thm:closeformula}. The rest of the theorem follows from
Propositions~\ref{prop:analyticF1} and~\ref{prop:analyticF2}.
\qed

\section{Lower bound: Proof of
  Theorem~\ref{thm:lowerbound}}\label{sec:LowerBound} 

Recall that an element $g$ in $\GL_n(q)$ is said to be a \emph{cyclic
matrix} if the 
characteristic polynomial of $g$ is equal to its minimum polynomial,
see~\cite{NP}.  If $g$ is a cyclic
matrix, then (see~\cite[Theorem~$2.1(3)$]{NP}) the group $C_{\GL_n(q)}(g)$ is
abelian and by~\cite[Corollary~$2.3$]{NP} we have
$|C_{\GL_n(q)}(g)|\leq q^n$. Thus the groups in
$\mathcal{N}_n(q)$ (see Definition~\ref{def:centr}) are abelian of
order at most $q^n$.  
 
Cyclic matrices of $\GL_n(q)$ are well-studied  (see~\cite{FNP,NP})
and  in particular Wall (see~\cite[page~$2$]{FNP}) proved that the
proportion $c_{\GL}(n,q)$ of cyclic matrices in 
$\GL_n(q)$ satisfies
$$\left|c_{\GL}(n,q)-\frac{1-q^{-5}}{1+q^{-3}}\right|\leq \frac{1}{q^n(q-1)}.$$ 
Thus 
\begin{equation}\label{eq:cyclicm}c_{\GL}(n,q)\geq
\frac{1-q^{-5}}{1+q^{-3}}-\frac{1}{q^n(q-1)}>
1-q^{-3}-q^{-5}+q^{-6}-q^{-n},
\end{equation} 
where the second inequality is obtained by expanding
$(1-q^{-5})/(1+q^{-3})$ in powers of $q$ and by noticing that
$1/q^n(q-1)\leq 1/q^n$. Using this remarkable  
result, we easily obtain 
Theorem~\ref{thm:lowerbound}.

\medskip

\noindent\emph{Proof of Theorem~\ref{thm:lowerbound}.} Let $\mathcal{C}_n(q)$ denote the set of cyclic matrices of
  $\GL_n(q)$ and  $X=\{(g,C)\mid g \in
  \mathcal{C}_n(q),C\in\mathcal{N}_n(q), g \in C\}$. We claim that
  every element $g$ of $\mathcal{C}_n(q)$ lies in a unique element
  of $\mathcal{N}_n(q)$. Indeed, assume $g\in C_1,C_2$, for some
  $C_1,C_2\in\mathcal{N}_n(q)$, and let $g_1,g_2$ be cyclic matrices
  such 
  that $C_i=C_{\GL_n(q)}(g_i)$, for $i=1,2$. As $C_1,C_2$ are abelian and
  $g\in C_1,C_2$, we get $C_1,C_2\subseteq
  C_{\GL_n(q)}(g)$. Similarly, as $C_{\GL_n(q)}(g)$ is abelian and
  $g_i\in C_i\subseteq C_{\GL_n(q)}(g)$, we get  $
  C_{\GL_n(q)}(g)\subseteq C_i$ and $C_{\GL_n(q)}(g)=C_1=C_2$. 

Counting the size of the set $X$, we have 
\begin{eqnarray*}
q^nN_n(q)&=&q^n|\mathcal{N}_n(q)|=\sum_{C\in\mathcal{N}_n(q)}q^n\geq
\sum_{C\in \mathcal{N}_n(q)}|C\cap\mathcal{C}_n(q)|=|X|\\
&=&\sum_{g\in \mathcal{C}_n(q)}|\{C\in\mathcal{N}_n(q)\mid g \in
C\}|=\sum_{g\in \mathcal{C}_n(q)}1=|\mathcal{C}_n(q)|=|\GL_n(q)|c_{\GL}(n,q).
\end{eqnarray*}
Now Equation~\ref{eq:cyclicm} yields $N_n(q)\geq
q^{-n}|\GL_n(q)|(1-q^{-3}-q^{-5}+q^{-6}-q^{-n})$. 

It remains to prove that $\omega(\GL_n(q))\geq N_n(q)$. Let
$C_1,\ldots,C_r$ be the distinct elements of $\mathcal{N}_n(q)$, with
$r=N_n(q)$. Let $g_i$ be a cyclic matrix in $\GL_n(q)$ such that
$C_i=C_{\GL_n(q)}(g_i)$, for $i=1,\ldots,r$. Set $S=\{g_i\mid 1\leq i\leq
r\}$. We claim that, if $i\neq j$, then the group elements
$g_{i}$ and $g_{j}$ of $S$ do not commute. If 
$g_ig_j=g_jg_i$, then $g_j\in C_{\GL_n(q)}(g_i)=C_i$, whereas we showed
above that $C_j$ is the unique element of $\mathcal{N}_n(q)$
containing $g_j$. This yields $\omega(\GL_n(q))\geq |S|=r=N_n(q)$ and
thus the theorem follows.\qed

\section{Upper bound: idea of the proof}\label{sec:ideaoftheproof}

In the rest of this paper, we determine an upper bound for
$\omega(\GL_n(q))$ (and hence for $N_n(q)$ by
Theorem~\ref{thm:lowerbound}). Also, for $q>n$, we prove that
$N_n(q)=\omega(\GL_n(q))$ and we obtain 
an exact formula for 
$N_n(q)$. Before going into more detail, in this section we 
briefly describe the method that is used. First, our results rely on
this elementary observation. 
\begin{lemma}\label{lemma:elementary}
Let $G$ be a group and  $\mathcal{A}$ be a collection of
abelian subgroups of $G$ such that $G=\cup_{A\in\mathcal{A}}A$. We
have $\omega(G)\leq |\mathcal{A}|$.   
\end{lemma}
\begin{proof}
Let $S$ be a pairwise non-commuting set. Since $A\in \mathcal{A}$ is
abelian, we get $|S\cap A|\leq 1$. As 
$G=\cup_{A\in \mathcal{A}}A$, we obtain  $|S|\leq
|\mathcal{A}|$. Thus the result follows.
\end{proof}

Lemma~\ref{lemma:elementary} can be used effectively to
obtain upper bounds for $\omega(G)$. As an example we derive  Brown's
upper bound for $\omega(\Sym(n))$ mentioned in Subsetion~\ref{sec:sub},
see~\cite[Theorem~1~(1)]{Brown}.

\begin{proposition}\label{Symn:elementary}
There exists a constant $b$, which does not depend on $n$,
such that $\omega(\Sym(n))\leq b(\log\log n)(n-2)!$.
\end{proposition}
\begin{proof}
By~\cite[Theorem~$2$]{Dixon}, the number of maximal abelian subgroups
of $\Sym(n)$ is at most $b(\log\log n)(n-2)!$, for some constant $b$
not depending on $n$. Thus the proposition follows from
Lemma~\ref{lemma:elementary}. 
\end{proof}

Unfortunately, there is no natural description (as in $\Sym(n)$) for  the
maximal abelian subgroups  of $\GL_n(q)$. So, it looks particularly
difficult to give an upper 
bound for the number of all maximal abelian subgroups of
$\GL_n(q)$. (For some results on the number of maximal abelian
subgroups with trivial unipotent radical in Chevalley groups, we refer
the reader to~\cite{Vdovin}.) We overcome this difficulty by focusing
only on the  subfamily $\mathcal{A}_n(q)$ of abelian subgroups defined
in Definition~\ref{def:An} which is large
enough to cover all the group elements as will be proved in
Proposition~\ref{prop:coveringproperty}.  
This leads to an upper bound for
$\omega(\GL_n(q))$. For $q>n$, we construct a
pairwise non-commuting set of size $|\mathcal{A}_n(q)|$ and so we obtain
an explicit formula for $\omega(\GL_n(q))$.

\begin{remark}\label{rm:0}{\rm We note that, for general $q$ and $n$,
  centralisers of cyclic matrices 
  do not cover all the elements in $\GL_n(q)$. Here we simply give an
  example for $\GL_4(2)$. Consider the matrix
\[
x=\left(
\begin{array}{cccc}
1&0&0&0\\
0&1&1&0\\
0&0&1&1\\
0&0&0&1\\
\end{array}
\right).
\]
With an easy computation, it is easy to check that 
\[
C=C_{\GL_4(2)}(x)=\left\{
\left(
\begin{array}{cccc}
1&0&0&a\\
b&1&c&d\\
0&0&1&c\\
0&0&0&1\\
\end{array}
\right)\mid a,b,c,d\in \mathbb{F}_2
\right\}
\]
and that $C$ has order $16$. In particular, $C$ consists of unipotent
elements. A unipotent matrix $u$ is cyclic if and only if $u$ has
minimum polynomial $(t-1)^4$, that is, $u-1$
has rank $3$. Now, it is easy to see that if $c\in C$, then $c-1$ has
rank at most $2$. Therefore, $C$ contains no cyclic matrix and
hence $x$ is not contained in the centraliser of a cyclic matrix. 

A similar example can be constructed for every $q>2$. Namely, consider
the matrix  
\[
x=\left(
\begin{array}{cc}
D&0\\
0&U\\
\end{array}
\right)
\]
in $\GL_{2q-1}(q)$, where $D$ is a $(q-1)\times (q-1)$-diagonal
matrix with distinct eigenvalues and $U$ is a $(q\times q)$-cyclic matrix with
minimum polynomial $(t-1)^q$ (that is, a regular unipotent element of
$\GL_q(q)$). It is possible to show that $x$ is not contained in the
centralizer of a cyclic matrix of $\GL_{2q-1}(q)$.}
\end{remark}

\section{Conjugacy classes and centralisers in $\GL_n(q)$}\label{sec:notation}

In this section, we introduce some notation and some well-known
results that are going to be used throughout the rest of the paper.

Let $k$ be a field with $q$ elements, $V$ be $k^n$ and
$k[t]$ be the polynomial 
ring with coefficients in $k$. Now, each element $g$ of 
$\GL_{n}(q)$ acts on the vector space $V$ and hence defines
on $V$ a  $k[t]$-module structure by setting $tv=gv$. We denote
this $k[t]$-module by $V_g$. For instance, it is easy to see that  $g$
is a cyclic matrix if and 
 only $V_g$ is a cyclic $k[t]$-module.
Clearly, any two elements $g,h$ of $\GL_{n}
(q)$ are conjugate if and only if $V_g$ and $V_h$ are isomorphic
$k[t]$-modules. 

For each element $g$ of $\GL_n(q)$, there exist unique
$s,u\in \GL_n (q)$ 
such that $g=su=us$, where $s$ is \emph{semisimple} and $u$ is
\emph{unipotent} (see~\cite[Section~$1.4$]{Carter}). We 
call $s$ (respectively $u$) the semisimple (respectively unipotent)
part of $g$. 

A unipotent element $u$ of $\GL_n(q)$ is said to be a
\emph{regular unipotent}  element if $u-1$ has rank $n-1$, that is  $u$
has minimum polynomial $(t-1)^n$ and so $u$ is a cyclic matrix. In
particular,  regular unipotent 
elements of $\GL_n(q)$ form a $\GL_n(q)$-conjugacy class. In the following
lemma, we collect some well-known information on the centraliser and
normaliser of a regular unipotent element. 

\begin{lemma}\label{lemma:regularunipotent}
Let $u$ be a regular unipotent element of $\GL_n(q)$, for $n\geq
2$. The group  $C=C_{\GL_n(q)}(u)$ is abelian of order $(1-q^{-1})q^n$ and
$N_{\GL_{n}(q)}(C)$ has order $(1-q^{-1})^2q^{2n-1}$.
\end{lemma}

\begin{proof}Set $v=u-1$. Since $u$ is a regular unipotent element,
  the element $v$ is a nilpotent matrix of rank $n-1$ with minimal
  polynomial $t^{n}$. Also, $C_{\GL_n(q)}(u)=C_{\GL_n(q)}(v)$. The
  elements centralizing $v$ are the
  isomorphisms of the  
  $k[t]$-module $V_v$. Since $V_v\cong k[t]/(t^n)$ is a
  uniserial module,
  it is readily seen (see for example~\cite[page~$265$]{NP}) that
  $\mathrm{End}_{k[t]}(V_v)$ is a polynomial ring in $v$ isomorphic to $k[t]/(t^n)$. Therefore,
  $\mathrm{End}_{k[t]}(V_v)=\langle 1,v,\ldots,v^{n-1}\rangle$ is abelian. Since the ideals of
  $k[t]/(t^n)$ are in one-to-one correspondence with the ideals of
  $k[t]$ that contain $(t^n)$ and $(t)$ is the unique maximal ideal
  containing $(t^n)$, it follows that
  $\mathrm{End}_{k[t]}(V_v)$ is a local ring with maximal ideal
  $(v)=\langle v,\ldots,v^{n-1}\rangle$ and every element of $(v)$ is nilpotent. In particular, the
  element  $x=\sum_{i=0}^{n-1}a_iv^i$ of $\mathrm{End}_{k[t]}(V_v)$ is
  invertible if and only if $x\notin (v)$, that is $a_0\neq 0$. This shows that
  $C$ is abelian of order $q^n-q^{n-1}=(1-q^{-1})q^n$.  

Let $x=\sum_{i=0}^{n-1}a_iv^i$ be in $\mathrm{End}_{k[t]}(V_v)$.  We
claim that $x$ is a regular unipotent element if and only if $a_0=1$
and $a_1\neq 0$. Assume first that $x$ is a regular unipotent
element. Thus $x-1$ is a nilpotent element with minimum polynomial
$t^n$. Now, $x-1$ is nilpotent if and only if $a_0-1=0$, that is
$a_0=1$. Also, as $(v^2)^{m}=0$ for every $m\geq n/2$, we obtain that
 $x-1$ is not a multiple of $v^2$, that is $a_1\neq 0$. Conversely, assume
that $a_0=1$ and $a_1\neq 0$. In particular, $x-1=vy$, where by the
previous paragraph $y$ is an invertible element of
$\mathrm{End}_{k[t]}(V_v)$. So, $(x-1)^{n-1}=v^{n-1}y^{n-1}\neq 0$ and
$x-1$ has minimum polynomial $v^{n-1}$. Thus $x$ is a regular
unipotent element.  This yields that 
  $C$ contains $q^{n-1}-q^{n-2}=(1-q^{-1})q^{n-1}$ regular unipotent
  elements. Since the regular
  unipotent elements form a $\GL_n(q)$-conjugacy class, $C$
  contains $(1-q^{-1})q^{n-1}$ regular unipotent elements and
  $C=C_{\GL_n(q)}(u')$ for each regular unipotent element $u'\in C$, we 
have that $|N_{\GL_n(q)}(C)|/|C|=(1-q^{-1})q^{n-1}$ and
$|N_{\GL_n(q)}(C)|=(1-q^{-1})^{2}q^{2n-1}$.  
\end{proof}

Let $d,m\geq 1$ be integers such that  $n=dm$ and  $E$ be a field
extension  over 
$k$ of degree $d$. As $E$ is a $k$-vector space of dimension $d$ and
$d$ divides $n$, we have that $k^n$ is isomorphic to $E^{m}$ as
$k$-vector spaces. Under
this isomorphism, the group 
$\GL_{m}(q^d)$ embeds  into a subgroup of $\GL_{n}(q)$, which we
still denote by $\GL_{m}(q^d)$. This does not cause any confusion
because all fields of order $q^d$ are isomorphic, and therefore different
embeddings give rise to subgroups which are conjugate.

We recall that, given a group $G$,  a $G$-module $V$ is said to be
\emph{indecomposable} if 
$V\neq 0$ and if it is impossible to express $V$ as a direct sum of
two non-trivial $G$-submodules.  In the next lemma we determine the
elements $g$ of $\GL_n(q)$ such that $V_g$ is indecomposable.

\begin{lemma}\label{lemma:indecomposable}
Let $g$ be in $\GL_n(q)$ such
that $V_g$ is an indecomposable $k[t]$-module, where $g$ has
semisimple part $s$ and unipotent part $u$. Then $g$ is a cyclic
matrix with minimum
polynomial $f^m$, for some irreducible polynomial $f$ of degree $d$
with $dm=n$. Replacing $g$ by a conjugate if
necessary, $g\in \GL_{m}(q^d)$,
the element $s$ is a scalar matrix of  $\GL_{m}(q^d)$ corresponding to
a generator of $\mathbb{F}_{q^d}$ and the element
$u$ is a regular  unipotent element of $\GL_{m}(q^d)$. In particular,
$C_{\GL_n(q)}(g)$ is abelian of order $(1-q^{-d})q^{dm}$.  
\end{lemma}
\begin{proof}
Since $k[t]$ is a principal ideal domain, we have that the
$k[t]$-module $V_g$ is a direct sum of cyclic modules of the form
$k[t]/(f^m)$, where $f$ is a monic irreducible polynomial of $k[t]$ and
$m\geq 1$. As $V_g$ is indecomposable, we obtain that $V_g\cong
k[t]/(f^m)$, for some irreducible polynomial
$f=t^d-\sum_{i=1}^{d}a_it^{i-1}$ of degree $d$ and
$n=dm$. Let $J(f)$ denote the companion matrix for the polynomial $f$
\[
J(f)=\left(
\begin{array}{ccccc}
0     &1  &0&\cdots &0\\
0     &0  &1&\cdots &0\\
\cdots&   & &       &\cdots\\
0     &0  &0&\cdots &1\\
a_1   &a_2& &\cdots &a_d\\
\end{array}
\right)
\]
and let 
\[
J_m(f)=\left(
\begin{array}{ccccc}
J(f)&I_d   &0  &\cdots&0 \\
0     &J(f)&I_d&\ddots&\vdots\\
\vdots & \ddots   &\ddots   &\ddots      &\\
\vdots     &    &\ddots   &J(f)&I_d\\
0     &\cdots  &\cdots& 0 & J(f)
\end{array}
\right)
\]
with $m$ diagonal blocks $J(f)$. By construction the characteristic
polynomial of the block matrix $J_m(f)$ equals
$f^m$. Also,~\cite[Example~$1$, page~$140$]{macdonald} shows that $f^m$
is the minimum polynomial of $J_m(f)$. Therefore $J_m(f)$ is a cyclic
matrix. As $V_g$ and $V_{J_m(f)}$ are $k[t]$-modules isomorphic to
$k[t]/(f^m)$, we obtain that $g$ is
conjugate to $J_m(f)$ and so $g$ is a cyclic matrix with minimum
polynomial $f^m$. Thus we may assume that
$g=J_m(f)$. So, $s$ is obtained from $J_m(f)$ by replacing
the $d\times d$-identity matrix $I_d$ with the $d\times d$-zero matrix
$0$. Similarly, $u$ is obtained from $J_m(f)$ by replacying
the $d\times d$-matrix $J(f)$ with the $d\times d$-identity matrix
$I_d$.  

Now, the centraliser of the cyclic matrix $J(f)$ in the
algebra of $d\times d$-matrices over $k$ is a polynomial
algebra isomorphic to $k[t]/(f)$. Since $f$ is irreducible, $k[t]/(f)$
is a field of size $q^d$. Hence $C_{\GL_d(q)}(J(q))$ is a cyclic group of order
$q^d-1$ isomorphic to the multiplicative group of a field of size
$q^d$ and, since $f$ is irreducible, $J(f)$ corresponds to a generator
in this field.  Under this 
identification, $J_m(f)$ is an element of $\GL_m(q^d)$, $s$ 
is a scalar matrix correspoding to a generator of $\mathbb{F}_{q^d}$
and $u$ is a regular unipotent matrix. The rest of 
the lemma follows from Lemma~\ref{lemma:regularunipotent}. 
\end{proof}

\begin{corollary}\label{cor:cortolemma}
Let $g_1$ and $g_2$ be in $\GL_n(q)$ such that $V_{g_1}$ and $V_{g_2}$ are
indecomposable $k[t]$-modules. Set $C_{g_i}=C_{\GL_n(q)}(g_i)$, for
$i=1,2$. The following are equivalent:
\begin{description}
\item[$(i)$]$C_{g_1}$ is conjugate to $C_{g_2}$;
\item[$(ii)$]$g_1$ and $g_2$ have minimum polynomials $f_{g_1}^m$
  and $f_{g_2}^m$, for some irreducible polynomials $f_{g_1},f_{g_2}$
  of degree $d$ with $dm=n$; 
\item[$(iii)$]the $C_{g_1}$-module $V$ is isomorphic to the
  $C_{g_2}$-module $V$.
\end{description}
\end{corollary}
\begin{proof}
By Lemma~\ref{lemma:indecomposable}, $g_i$ is a cyclic matrix with
minimum polynomial   $f_{g_i}^{m_{g_i}}$, for some irreducible polynomial
$f_{g_i}$ of degree $d_{g_i}$  with $d_{g_i}m_{g_i}=n$ (for $i=1,2$).  
Assume $C_{g_1}$ is conjugate to $C_{g_2}$. 
By Lemma~\ref{lemma:indecomposable}, $|C_{g_1}|=(1-q^{-d_{g_1}})q^{d_{g_1}m_{g_1}}$
and $|C_{g_2}|=(1-q^{-d_{g_2}})q^{d_{g_2}m_{g_2}}$. As
$|C_{g_1}|=|C_{g_2}|$, we have $d_{g_1}=d_{g_2}$ and
$m_{g_1}=m_{g_2}$. Thus Part~$(i)$ implies 
Part~$(ii)$.

Assume Part~$(ii)$. Let $s_1$ and $s_2$ be the semisimple parts of
$g_1$ and $g_2$, respectively. Similarly, let $u_1$ and $u_2$ be the
unipotent parts 
of $g_1$ and $g_2$, respectively. By Lemma~\ref{lemma:indecomposable},
replacing $g_1$ and $g_2$ by a conjugate
if necessary, we may assume that $g_1,g_2\in \GL_m(q^d)$, $s_1$ and $s_2$
are scalar matrices corresponding to generators of
$\mathbb{F}_{q^d}$ and $u_1,u_2$ are regular unipotent elements of
$\GL_{m}(q^d)$. Therefore
$C_{g_i}=C_{\GL_n(q)}(g_i)=C_{\GL_m(q^d)}(g_i)=C_{\GL_m(q^d)}(u_i)$, 
for $i=1,2$. Since regular unipotent elements form a
$\GL_m(q^d)$-conjugacy class, we obtain that $u_1$ is conjugate to
$u_2$ in $\GL_m(q^d)$ and so $C_{g_1}=C_{\GL_m(q^d)}(u_1)$ is conjugate
to $C_{g_2}=C_{\GL_{m}(q^d)}(u_2)$ and Part~$(i)$ follows.

If $C_{g_1}$ is conjugate in $\GL_n(q)$ to $C_{g_2}$, then the
$C_{g_1}$-module $V$ is isomorphic to the $C_{g_2}$-module $V$. Thence
Part~$(i)$ implies Part~$(iii)$.

Conversely, if the $C_{g_1}$-module $V$ is isomorphic to the
$C_{g_2}$-module $V$, 
then there exists a group isomorphism $\varphi:C_{g_1}\to C_{g_2}$ and
a $k$-vector space isomorphism $\psi:V\to V$ such that
$(vg)\psi=(v\psi)(g^\varphi)$, for every $v\in V$ and $g\in C_{g_1}$. This yields
$g^\varphi=\psi^{-1}g\psi$, for every $g\in C_{g_1}$. Thence $C_{g_1}$ is conjugate to $C_{g_2}$
in $\GL_n(q)$ and Part~$(i)$ follows.
\end{proof}

The set of abelian subgroups 
$\{C_{\GL_n(q)}(g)\mid V_g\textrm{ indecomposable}\}$ of $\GL_n(q)$
plays a very important role in this paper. It is worth to point out
that, by Corollary~\ref{cor:cortolemma}, the conjugacy
classes in this  family of
subgroups are in one-to-one
correspondence with the ordered pairs of positive integers $(d,m)$ with
$n=dm$. We denote by 
\begin{equation}\label{def:A_dm}
\{A_{d,m}\}_{d,m}
\end{equation} 
a set of representatives
for these 
conjugacy classes. In particular, for $m=1$, the group $A_{d,1}$ is a
cyclic group, and actually $A_{d,1}$ is a maximal non-split torus of order
$q^d-1$ in  $\GL_d(q)$, usually called a Singer cycle.

\begin{lemma}\label{lemma:normalizer}
Let $d,m\geq 1$ be such that $n=dm$. The group $A_{d,m}$ is a maximal
abelian subgroup of $\GL_{n}(q)$ and
\[
|N_{\GL_n(q)}(A_{d,m})|=\left\{
\begin{array}{lcc}
d(1-q^{-d})^2q^{2dm-d}&&\textrm{if }m>1,\\
d(1-q^{-d})q^{d}&&\textrm{if }m=1.\\
\end{array}
\right.
\]
\end{lemma}
\begin{proof}
By definition of $A_{d,m}$, there exists an element $g=su$  of
$A_{d,m}$ such that 
$A_{d,m}=C_{\GL_{n}(q)}(g)$, where $s$ is the semisimple part of $g$ and
$u$ is the unipotent part of $g$. By Lemma~\ref{lemma:indecomposable},
 we may choose $A_{d,m}$ such that $A_{d,m}\subseteq \GL_{m}(q^d)$ and
 we may assume that $s$ is a scalar matrix of $\GL_{m}(q^d)$ of order $q^d-1$. 

By Lemma~\ref{lemma:indecomposable}, $A_{d,m}$ is abelian. Let  $A$ be
an abelian subgroup of $\GL_{n}(q)$ containing $A_{d,m}$ and $x$ be
in $A$. Since $A$ is abelian, $x$ commutes with $g$ and so $x\in
A_{d,m}$. This yields that $A_{d,m}$ is a maximal abelian subgroup of
$\GL_{n}(q)$.  

Let $N$ be the normaliser in $\GL_{n}(q)$ of $A_{d,m}$ and  $x$ be in
$N$. Since $\langle s\rangle$ is a normal Hall subgroup of $A_{d,m}$,
we get that $x$ normalises $\langle s\rangle$.  So $x$ normalises the
subgroup of
scalar matrices of $\GL_{m}(q^d)$. Thence, $x$ 
acts as a Galois automorphism on the field $k[s]$ of order $q^d$. This
shows that 
$|N:N\cap \GL_{m}(q^d)|=d$. If $m=1$, then $\GL_{1}(q^d)=A_{d,1}$
and $|N|=d(q^d-1)$. If $m>1$, 
then  Lemma~\ref{lemma:regularunipotent} yields that $N\cap
\GL_{m}(q^d)$ has order $(1-q^{-d})^2q^{2dm-d}$.  
\end{proof}

\section{The family $\mathcal{A}_n(q)$ and the upper bound for
  $\omega(\GL_n(q))$}\label{sec:thefamily}

Finally, we are ready to study the family $\mathcal{A}_n(q)$ of abelian
subgroups of $\GL_n(q)$ (given in Definition~\ref{def:An})
necessary in order to obtain an upper bound on the size of
$\omega(\GL_n(q))$. For convenience, we state 
Definition~\ref{def:An} again. 

\noindent\textbf{Definition~\ref{def:An}. }
Let $\mathcal{A}_n(q)$ be the set of abelian subgroups $A$ of $\GL_n(q)$
such that the $A$-module $V$ has a decomposition $V_1\oplus\cdots
  \oplus V_r$ into indecomposable $A$-modules satisfying the following
  properties: 
\begin{description}
\item[$(i)$]$A=A_1\times \cdots\times A_r$, where $A_i\subseteq \GL
  (V_i)$;
\item[$(ii)$]for $i=1,\ldots,r$, we have $A_i=C_{\GL
  (V_i)}(a_i)$, for some element $a_i\in \GL (V_i)$ such
  that $(V_i)_{a_i}$ is an 
  indecomposable $k[t]$-module.
\end{description}
We show in Proposition~\ref{prop:numberA_n} that, for $q>2$,  the
elements of $\mathcal{A}_n(q)$ are maximal abelian subgroup of
$\GL_n(q)$.
From the definition of $\mathcal{A}_n(q)$ we get at once
Theorem~\ref{thm:sumupthms}~$(a)$. 
\begin{proposition}\label{prop:coveringproperty}
$\GL_n(q)=\cup_{A\in \mathcal{A}_n(q)}A$. 
\end{proposition}
\begin{proof}
Given $x$ in $\GL_n(q)$, consider a decomposition of
$V_x=V_1\oplus\cdots\oplus V_r$ into
indecomposable $k[t]$-modules. The action of $x$ on $V_i$ is given
by some element $a_i\in \GL (V_i)$. 
By Lemma~\ref{lemma:indecomposable},
$A_i=C_{\GL(V_i)}(a_i)$ is abelian. Now, $x\in A=A_1\times \cdots
\times A_r$ and $A\in \mathcal{A}_n(q)$. 
\end{proof}

The following definition is necessary in order to have a natural set of
labels for the elements in $\mathcal{A}_n(q)$ (see Lemma~\ref{lemma:labels}).
\begin{definition}{\rm
We denote by
$\Phi$ the set of functions from $\{(d,m)\mid d,m\geq
1\}$ to $\mathbb{N}$. Also, we write 
$\Phi_n$ for the subset of $\Phi$ containing the functions $\mu$ such that
$n=\sum_{d,m}dm\mu(d,m)$.} 
\end{definition}
For instance, $\Phi_1$ contains only one
element, namely the function $\mu$ defined by $\mu(1,1)=1$ and $\mu(d,m)=0$, for
$m>1$ or $d>1$. 

\begin{lemma}\label{lemma:labels}The conjugacy classes of subgroups in
  $\mathcal{A}_n(q)$ 
  are in one-to-one correspondence with the elements of $\Phi_n$.
\end{lemma}
\begin{proof}
We define a bijection $\theta$ from $\Phi_n$ to the set of conjugacy
classes of subgroups in $\mathcal{A}_n(q)$. Let $\mu$ be in $\Phi_n$. For
each $d,m$ and $i$ with $1\leq 
i\leq \mu(d,m)$, let $W_{d,m,i}$ be a $k$-subspace of $V$ of dimension $dm$
such that $V=\oplus_{d,m,i}W_{d,m,i}$. Note that this is possible
because $\dim_k(V)=n=\sum_{d,m}dm\mu(d,m)$. Consider
$A_{d,m}^{(i)}\leq\GL (W_{d,m,i})$, with $A_{d,m}^{(i)}=A_{d,m}$ as in
Equation~\ref{def:A_dm}, and  set 
$A=\prod_{d,m,i}A_{d,m}^{(i)}$. By construction, $A\in
\mathcal{A}_n(q)$. Define 
$\theta(\mu)$ to be 
the conjugacy class containing $A$. Now, let $V=\oplus_{k=1}^sM_k$ be
another decomposition of $V$ into a direct sum of non-zero
indecomposable $A$-submodules. By the Krull-Schmidt
theorem~\cite[Theorem~$14.5$]{CuRe},  there exists a bijective
function $f$ between the set of indices $\{(d,m,i)\mid 1\leq i\leq
\mu(d,m)\}$ and $\{1,\ldots,s\}$ such that $W_{d,m,i}\cong
M_{f(d,m,i)}$. Thus, Corollary~\ref{cor:cortolemma} yields that $\mu$
is uniquely determined from the conjugacy class of $A$ in $\GL_n(q)$,
that is, $\theta$ is injective.

The map $\theta$ is surjective by the definition of
$\mathcal{A}_n(q)$. 
\end{proof}

Given $\mu\in \Phi_n$, we denote by 
\begin{equation}\label{def:Amu}
A_\mu
\end{equation} a representative of the
conjugacy class in $\mathcal{A}_n(q)$ corresponding to $\mu$ in
$\Phi_n$. We note that if $\mu(d,m)=0$ for  
$m>1$, then $A_\mu$ is a torus in $\GL_n(q)$. In particular,
every maximal torus of $\GL_n(q)$ is a member of $\mathcal{A}_n(q)$.  

Before  proving the main result of this paper, we need
first a definition  and then some preliminary lemmas. 

\begin{definition}\label{def:cyclicunip}
{\rm Let $\mu$ be in $\Phi_n$ and $q=2$. We
  say that $A_\mu$ has \emph{cyclic unipotent summand} if $\mu(1,x)=0$
  for all but at  most one value of $x$, and if $\mu(1,x)\neq 0$
  for $x=m$, say, then $\mu(1,m)=1$.
In particular, by Definition~\ref{def:An} and
  Lemma~\ref{lemma:indecomposable}, $V$ has at most one indecomposable 
  $A_\mu$-invariant summand $W$ such that the action of $A_\mu$ on $W$
  is given by the centralizer of a regular unipotent matrix (which is a
  cyclic matrix).}   
\end{definition}

\begin{lemma}\label{lemma:uniquedecomposition}Let $\mu$ be in
  $\Phi_n$. The decomposition of  $V$  as direct
sum of  indecomposable $A_\mu$-modules is unique up to permutation of the
summands if and only if either $q\geq 3$, or $q=2$ and  $A_\mu$ has cyclic
unipotent summand.
\end{lemma}
\begin{proof}
Let $$V=\bigoplus_{ 
\scriptsize
\begin{array}{c}
d,m,\\
1\leq i\leq \mu(d,m)
\end{array}
\normalsize
}V_{d,m}^{(i)}$$ be an $A_\mu$-invariant direct decomposition
 in indecomposable modules labelled so that $\dim V_{d,m}^{(i)}=dm$ (see
 Lemmas~\ref{lemma:labels}). By
 Definition~\ref{def:An}~$(i)$, we have 
 $A_\mu=\prod_{d,m,i}A_{d,m}^{(i)}$, where $A_{d,m}^{(i)}\subseteq
 \GL(V_{d,m}^{(i)})$.

Assume $q\geq 3 $, or $q=2$ and $A_\mu$ has cyclic unipotent
summand. We show that 
the decomposition of $V$ as direct sum 
of indecomposable $A_\mu$-modules is unique, up to permutation of the
summands. 
Let $W$ be an
  indecomposable $A_\mu$-invariant summand of $V$. Now, the $A_\mu$-module
  $W$ is cyclic, that is, there exists $v\in W$ 
such that $W=\langle v\rangle_{A_\mu}$ (where $\langle v\rangle_{A_\mu}=\langle
va\mid a \in A_\mu\rangle$). Write $v=\sum_{d,m,i}v_{d,m}^{(i)}$, with
$v_{d,m}^{(i)}\in V_{d,m}^{(i)}$. We 
claim that 
\begin{equation}\label{eq:-1}
W=\bigoplus_{d,m,i}\langle
v_{d,m}^{(i)}\rangle_{A_\mu}.
\end{equation}  
The $A_\mu$-module $W$ is generated by $va=\sum_{d,m,i}v_{d,m}^{(i)}a$
(for $a\in A_\mu$), where $v_{d,m}^{(i)}a\in \langle  
v_{d,m}^{(i)}\rangle_{A_\mu}$. Therefore, $W\subseteq 
\sum_{d,m,i} \langle v_{d,m}^{(i)}\rangle_{A_\mu}$. Conversely, as
$\langle v_{d,m}^{(i)}\rangle_{A_\mu}\cap\langle
v_{d',m'}^{(i')}\rangle_{A_\mu}\subseteq V_{d,m}^{(i)}\cap
V_{d',m'}^{(i')}=0$ for $(d,m,i)\neq 
(d',m',i')$, it suffices to prove that 
$v_{d,m}^{(i)}\in W$ for every $d,m,i$. By
Definition~\ref{def:An}~$(ii)$ and Lemma~\ref{lemma:indecomposable},
the group $A_{d,m}^{(i)}$ contains a  scalar matrix $s_{d,m}^{(i)}$
of $\GL_{m}(q^d)$ 
corresponding to a generator of $\mathbb{F}_{q^d}$ if $d\geq 2$, and to
a primitive element of $\mathbb{F}_q$ if $d=1$. In particular,
$s_{d,m}^{(i)}\neq 1$ if $(q,d)\neq (2,1)$.  

Assume first that $(q,d)\neq (2,1)$. Since $s_{d,m}^{(i)}$ acts as the identity
matrix on $V_{d',m'}^{(i')}$ (for $(d',m',i')\neq (d,m,i)$), we get
$vs_{d,m}^{(i)}=\sum_{(d',m',i')\neq
  (d,m,i)}v_{d',m'}^{(i')}+v_{d,m}^{(i)}s_{d,m}^{(i)}$. Therefore, 
$v_{d,m}^{(i)}(s_{d,m}^{(i)}-1)=vs_{d,m}^{(i)}-v\in W$. As
$s_{d,m}^{(i)}$ acts as a 
non-identity scalar matrix on $V_{d,m}^{(i)}$, we obtain
$v_{d,m}^{(i)}\in W$. This yields that if $(q,d)\neq (2,1)$, then
$v_{d,m}^{(i)}$ lies in $W$ for every $m$ and $i$. By hypothesis on
$q$ and $A_\mu$ and by Definition~\ref{def:cyclicunip}, we obtain that
all but possibly 
one summand  of $v$ lies in $W$. The exceptional case $(q,d)=(2,1)$
occurs only if 
$q=2$ and $v_{1,m}^{(1)}$ is the only 
summand of $v$ that is not covered by the argument in this paragraph; since
$v\in W$, we get  also in this case that $v_{1,m}^{(1)}\in W$ and
hence that every summand of $v$ lies in 
$W$. Our claim is now proved. 

Since $W$ is indecomposable, from Equation~\ref{eq:-1} we have $W=\langle
v_{d,m}^{(i)}\rangle_{A_\mu}\subseteq V_{d,m}^{(i)}$ for 
some $d,m,i$. Since $W$ is an $A_\mu$-invariant summand and  $V_{d,m}^{(i)}$ is
indecomposable, $W=V_{d,m}^{(i)}$. As $W$ is an arbitrary
indecomposable summand of $V$, we get that
$\{V_{d,m}^{(i)}\}_{d,m,i}$ are the only indecomposable summands of $V$
and the decomposition is unique.

Conversely, assume that $q=2$ and $A_\mu$ does not have a cyclic unipotent
summand, that is, $\mu(1,m)\geq 2$ for some $m$,
or  $\mu(1,m_1)=\mu(1,m_2)=1$ with $m_1\neq m_2$. Let $V_1$
and $V_2$ be two distinct $A_\mu$-invariant indecomposable direct
summands of $V$ 
isomorphic to $V_{1,m}$ (if $\mu(1,m)\geq 2$) or isomorphic to
$V_{1,m_1}$ and $V_{1,m_2}$ (if $\mu(1,m_1)=\mu(1,m_2)=1$). By
Lemma~\ref{lemma:indecomposable}, Equation~\ref{def:A_dm} and
Definition~\ref{def:An}, $V_i$ is
an $A_i$-module, where $A_i=C_{\GL (V_i)}(u_i)$ and  $u_i$ is  a regular
unipotent matrix of $\GL(V_i)$. Let $v_{1,1},\ldots, v_{1,r_1}$ 
(respectively, $v_{2,1},\ldots,v_{2,r_2}$) be a $k$-basis of $V_1$
(respectively, $V_2$) such that $v_{i,j}^{u_i}=v_{i,j}+v_{i,j-1}$
(for $1<j\leq r_{i}$) and $v_{i,1}^{u_i}=v_{i,1}$ for $i=1,2$.  Define
$V_2'=\langle v_{2,1},\ldots,v_{2,r_2-1},
v_{1,1}+v_{2,r_2}\rangle$. Clearly, $V_1\oplus V_2=V_1\oplus V_2'$. We
claim that $V_{2}'$ is an indecomposable $A_\mu$-invariant summand of
$V$. Since $u_2$ is a cyclic matrix, $\End_{k\langle u_2\rangle}(V_2)$ is a
polynomial algebra in $u_2$. Therefore, in order to show that $V_2'$ is
an $A_\mu$-invariant summand of $V$, it suffices to show that $V_2'$ is
$\langle u_2\rangle$-invariant, which is clear from the definition of
$V_2'$ and from the action of $u_2$ on $V_2$. 

As $V_{2'}\subseteq V_{1}\oplus V_2$,  $V_2'\neq
V_{1}$ and $V_2'\neq V_2$, we obtain that the decomposition of $V$ as
direct sum of 
indecomposable  $A_\mu$-modules is not unique. 
\end{proof} 

We give a definition which is needed in
Proposition~\ref{prop:numberA_n} and in Section~\ref{sec:F}.

\begin{definition}\label{def:stab}{\rm Let $\mu $ be in $\Phi_n$ and
  $V=V_1\oplus \cdots 
  \oplus V_r$ be an $A_\mu$-invariant decomposition of $V$ in
  indecomposable modules. We write $\Stab(V,\mu)$ for the subgroup of
   $\GL_n(q)$ preserving the direct decomposition $V_1\oplus\cdots\oplus V_r$ of
  $V$, that is, $\Stab(V,\mu)=\{g\in \GL_n(q)\mid V_i^g\in
  \{V_1,\ldots,V_r\}\textrm { for every }i\}$.}  
\end{definition}

We start by computing the size of the normaliser of a subgroup
$A_\mu$ and by proving that, if $q\geq 3$ or $q=2$ and  $A_\mu$ has
cyclic unipotent summand, then $A_\mu$ is a maximal
abelian subgroup  of 
$\GL_n(q)$, for $\mu \in\Phi$. 
\begin{proposition}\label{prop:numberA_n}Let $\mu$ be in
  $\Phi_n$. Then 
$|N_{\GL_n(q)}(A_\mu)\cap \Stab(V,\mu)|$ equals $$
\left(\prod_{d\geq  1}(d(1-q^{-d})q^d)^{\mu(d,1)}\mu(d,1)!\right)
\left(\prod_{d\geq 1,m\geq
  2}(d(1-q^{-d})^2q^{2dm-d})^{\mu(d,m)}\mu(d,m)!\right).$$
If  either $q\geq 3$, or $q=2$ and $A_\mu$ has cyclic unipotent summand, then
$N_{\GL_n(q)}(A_\mu)\subseteq \Stab(V,\mu)$, and  $A_\mu$ is a maximal
abelian subgroup of $\GL_n(q)$.
\end{proposition}

\begin{proof}
Let $\mu$ be in $\Phi_n$. Write $A_\mu=\prod_{d,m}A_{d,m}^{\mu(d,m)}$,
where $A_{d,m}$ is as defined in Equation~\ref{def:A_dm}. By
Definition~\ref{def:stab}, we have  
$$|N_{\GL_n(q)}(A_\mu)\cap
\Stab(V,\mu)|=\prod_{d,m}|N_{\GL_{dm}(q)}(A_{d,m})|^{\mu(d,m)}\mu(d,m)!.$$ 
Applying Lemma~\ref{lemma:normalizer}, the equality in the proposition follows. 

Assume that either $q\geq 3$, or that $q=2$ and $A_\mu$ has cyclic
unipotent summand. 
By Lemma~\ref{lemma:uniquedecomposition}, every element of
$\GL_n(q)$ normalising $A_\mu$ induces a permutation of the
indecomposable $A_\mu$-submodules of
$V$. Also, Corollary~\ref{cor:cortolemma} yields that indecomposable
$A_{\mu}$-submodules are isomorphic if and only if they correspond to
the same $d,m$. Therefore, $N_{\GL_n(q)}(A_\mu)\subseteq
\Stab(V,\mu)$. Furthermore, by
Lemma~\ref{lemma:normalizer}, $A_{d,m}$ is a maximal abelian subgroup
of $\GL_{dm}(q)$ and so $A_\mu$ is a maximal abelian subgroup of $\GL_n(q)$. 
\end{proof}

\begin{remark}\label{rk:00}
{\rm The converse of the last assertion of Proposition~\ref{prop:numberA_n}
is also true. Indeed, if $q=2$ and $A_\mu$ does not have cyclic
unipotent summand then $N_{\GL_n(q)}(A_\mu)$ contains an element $x$
that does not lie in $\Stab(V,\mu)$ and $\langle A_\mu,x\rangle$ is
abelian. Namely, in the notation of the last part of the proof of
Lemma~\ref{lemma:uniquedecomposition}, the element $x$ can be defined
to act as the identity on all summands of the decomposition $\oplus_iV_i$,
except those denoted $V_1\oplus V_2$. The action of $x$ on $V_1\oplus
V_2$ is given by $v_{i,j}^x=v_{i,j}$ except
that $v_{1,1}^x=v_{1,1}+v_{2,r_2}$.}
\end{remark}

Next, we prove Theorem~\ref{thm:sumupthms}~$(c)$.

\begin{theorem}\label{thm:exact}$\mathcal{N}_n(q)\subseteq\mathcal{A}_n(q)$
  with equality if and only if $q> n$. 
\end{theorem}

\begin{proof}
First we show that $\mathcal{N}_n(q)\subseteq \mathcal{A}_n(q)$. Let
$C$ be an element of $\mathcal{N}_n(q)$. By 
Definition~\ref{def:centr}, $C=C_{\GL_n(q)}(g)$ for some cyclic matrix
$g\in \GL_n(q)$. Consider a decomposition of
$V_g=V_1\oplus\cdots\oplus V_r$ into
indecomposable $k[t]$-modules. The action of $g$ on $V_i$ is given
by some element $a_i\in \GL (V_i)$. By
Lemma~\ref{lemma:indecomposable}, $A_i=C_{\GL(V_i)}(a_i)$ is abelian.
Now, $g\in A=A_1\times \cdots \times A_r$ and $A\in
\mathcal{A}_n(q)$. Replacing $C$ by a conjugate if necessary, we may
assume $A=A_\mu$, for $\mu\in \Phi_n$. Moreover since $g$ is cyclic it
follows that, if  $q=2$, then $A$ has  cyclic unipotent
summand. So, by Proposition~\ref{prop:numberA_n}, $A$
is a maximal abelian subgroup. As $g\in
A$, we have $A\subseteq C_{\GL_n(q)}(g)=C$. Since $C$ is abelian, we
have $A=C$ and  $C\in\mathcal{A}_n(q)$. 

Finally, in the rest of the proof we show that
$\mathcal{N}_n(q)=\mathcal{A}_n(q)$ if and only if $q>n$.

Assume $q>n$. As $\mathcal{N}_n(q)\subseteq \mathcal{A}_n(q)$, by
Lemma~\ref{lemma:labels}, it suffices to prove 
that for every $\mu\in\Phi_n$ there exists
a cyclic matrix $g_\mu\in A_\mu$ such that $A_\mu=C_{\GL_n(q)}(g_\mu)$. For every
$d,m\geq 1$ such that $dm\leq n$, let
$f_{{d,m,1}},\ldots,f_{{d,m,\mu(d,m)}}$ be  irreducible polynomials
 of degree $d$. Note that since $q>n$, we may choose $f_{{d,m,i}}$ so
 that the polynomials $(f_{{d,m,i}})_{d,m,i}$ are pairwise
 distinct. Let $$V=\bigoplus_{ 
\scriptsize
\begin{array}{c}
d,m,\\
1\leq i\leq \mu(d,m)
\end{array}
\normalsize
}V_{d,m}^{(i)}$$ be an $A_\mu$-invariant direct decomposition
 in indecomposable modules labelled so that $\dim V_{d,m}^{(i)}=dm$ (see
 Lemmas~\ref{lemma:labels}). By
 Definition~\ref{def:An}~$(i)$, we have 
 $A_\mu=\prod_{d,m,i}A_{d,m}^{(i)}$, where $A_{d,m}^{(i)}\subseteq
 \GL(V_{d,m}^{(i)})$. Also, by Definition~\ref{def:An}~$(ii)$,
 $A_{d,m}^{(i)}=C_{\GL(V_{d,m}^{(i)})}(h_{d,m}^{(i)})$, for some element
 $h_{d,m}^{(i)}\in \GL(V_{d,m}^{(i)})$ such that $V_{d,m}^{(i)}$ is an
 indecomposable 
 $\langle h_{d,m}^{(i)}\rangle$-module. By
 Lemma~\ref{lemma:indecomposable} and Definition~\ref{def:A_dm},
 $h_{d,m}^{(i)}$ is a cyclic matrix with minimum polynomial
 $p_{d,m,i}^m$, for some irreducible polynomial $p_{d,m,i}$ of degree $d$. Let
$g_{d,m}^{(i)}\in \GL(V_{d,m}^{(i)})$ be a cyclic matrix with
minimum polynomial $f_{d,m,i}^m$. Thence $V_{d,m}^{(i)}$ is an
indecomposable $\langle g_{d,m}^{(i)}\rangle$-module. Set
$$g_\mu=\bigoplus_{
\scriptsize
\begin{array}{c}
d,m,\\
1\leq i\leq \mu(d,m)\\
\end{array}
\normalsize}g_{d,m}^{(i)}.$$ 
Since $f_{d,m,i}$ and $p_{d,m,i}$ have both degree $d$, by
Corollary~\ref{cor:cortolemma} the groups 
$C_{\GL(V_{d,m}^{(i)})}(g_{d,m}^{(i)})$ and
$C_{\GL(V_{d,m}^{(i)})}(h_{d,m}^{(i)})=A_{d,m}^{(i)}$ are
conjugate. So, by Definition~\ref{def:Amu} and by 
construction, $g_\mu$ is conjugate to an element in $A_\mu$. Hence,
replacing $g_\mu$ by a conjugate if necessary, $g_\mu\in A_\mu$. The
characteristic 
polynomial of $g_\mu$ is $\prod_{d,m,i}f_{d,m,i}^{m}$. As the
polynomials $f_{{d,m,i}}$ 
are distinct, the characteristic polynomial of $g_\mu$ is
equal to its minimum polynomial. Thence $g_\mu$ is a cyclic
matrix. In particular, if $q=2$, then at most one of the $f_{d,m,i}$ is
 $t+1$, and hence at most one of the $p_{d,m,i}$ is $t+1$. So, $A_\mu$
has cyclic unipotent summand; hence by
Proposition~\ref{prop:numberA_n}, $A_\mu$ is a maximal abelian
subgroup. As  $C_{\GL_n(q)}(g_\mu)$ is abelian and $A_\mu\subseteq
C_{\GL_n(q)}(g_\mu)$, we get $A_\mu=C_{\GL_n(q)}(g_\mu)$. 

Assume $q\leq n$. We need to prove that $\mathcal{N}_n(q)\subset
\mathcal{A}_n(q)$. Let $\mu_0\in\Phi_n$ be the map defined by 
\begin{equation}\label{eq:mu0}
\mu_0(d,m)=\left\{
\begin{array}{ccl}
0&&\textrm{if }d\geq 2 \textrm{ or }m\geq 2,\\
n&&\textrm{if }d=1 \textrm{ and }m=1.\\
\end{array}
\right.
\end{equation}
By Definition~\ref{def:Amu}, $A_{\mu_0}$ is the group of
diagonal matrices, that is, the split torus
of size $(q-1)^n$.  Since there are only $q-1$ distinct
eigenvalues 
available,
any $g\in A_{\mu_0}$ has an eigenvalue with
multiplicity $\geq 2$. Therefore, $g$ is not a cyclic
matrix and $A_{\mu_0}$ contains no cyclic matrices. Thus $A_{\mu_0}\in
\mathcal{A}_n(q)\setminus \mathcal{N}_n(q)$. 
\end{proof}

This allows us to complete the proof of Theorem~\ref{thm:sumupthms}.

\begin{corollary}\label{cor:exact}
$\omega(\GL_n(q))\leq |\mathcal{A}_n(q)|$ with equality if
  and only if $q>n$. 
\end{corollary}
\begin{proof}
The inequality $\omega(\GL_n(q))\leq |\mathcal{A}_n(q)|$  follows from
Lemma~\ref{lemma:elementary} and
Proposition~\ref{prop:coveringproperty}.
If $q>n$, then by Theorem~\ref{thm:exact} we have
$\mathcal{A}_n(q)=\mathcal{N}_n(q)$. So, the inequality
$\omega(\GL_n(q))\geq |\mathcal{A}_n(q)|$ follows from
Theorem~\ref{thm:lowerbound}. Now  assume $q\leq n$. Let $\mu_0$
be the function in $\Phi_n$ defined in Equation~\ref{eq:mu0} and
$\mathcal{A}$ be the collection of subgroups $A$ of $\mathcal{A}_n(q)$ not
conjugate to $A_{\mu_0}$.  By Definition~\ref{def:Amu}, $A_{\mu_0}$ is
the group of 
diagonal matrices.  Since there are only $q-1$ distinct
eigenvalues  available,
every $g\in A_{\mu_0}$ has an eigenvalue with
multiplicity $\geq 2$. Therefore, every $g$ is contained in some $A$,
with $A\in \mathcal{A}$. Thus, from
Proposition~\ref{prop:coveringproperty}, we have $\GL_n(q)=\cup_{A\in
  \mathcal{A}}A$ and 
Lemma~\ref{lemma:elementary} yields $\omega(\GL_n(q))\leq
|\mathcal{A}|<|\mathcal{A}_n(q)|$. 
\end{proof}

Although Theorem~\ref{thm:lowerbound} and
Corollary~\ref{cor:exact} show that $N_n(q)$ and $\omega(\GL_n(q))$ are
bounded above by  $|\mathcal{A}_n(q)|$, the value and order of
magnitude of $|\mathcal{A}_n(q)|$ is not easy to establish from the
definition of 
$\mathcal{A}_n(q)$. Therefore, in the next section, we determine (for
$q>2$) a
closed simple formula for the generating function $F(t)$ of
$\{|\mathcal{A}_n(q)|/|\GL_n(q)|\}_{n\geq 1}$ (see
Theorem~\ref{thm:closeformula}). Also, in Section~\ref{sec:ana}, by
proving that $F(t)$ is 
analytic on a certain disk in the complex plane, we determine the asymptotic
behaviour of $|\mathcal{A}_n(q)|$, for $n\to\infty$.

\section{A generating function}\label{sec:F}
We start by defining two generating functions.
\begin{definition}\label{def:F}
{\rm
Let $F(t)=\sum_{n=1}^\infty a_nt^n$ be the generating function for the
proportion $a_n=|\mathcal{A}_n(q)|/|\GL_n(q)|$ and 
$\overline{F}(t)=\sum_{n=1}^\infty b_nt^n$ be the generating function
for $$b_n=\sum_{\mu \in \Phi_n }|N_{\GL_n(q)}(A_\mu)\cap \Stab(V,\mu)|^{-1}$$
(see Definitions~\ref{def:Amu} and~\ref{def:stab}).}   
\end{definition}

\begin{remark}\label{rk:6.2}{\rm 
By  Lemma~\ref{lemma:labels}, we 
get 
\begin{equation}\label{eq:1}
a_n=\sum_{\mu\in\Phi_n}|N_{\GL_n(q)}(A_\mu)|^{-1},
\end{equation}
and so, by Definition~\ref{def:F}, $a_n\leq b_n$ for every $n$. Furthermore, by
Proposition~\ref{prop:numberA_n}, $F(t)=\overline{F}(t)$ for 
$q\geq 3$. In particular, $F(t)$ and $\overline{F}(t)$ differ only
when  $q=2$, in  which case each coefficient $b_n$ of
$\overline{F}(t)$ 
is an upper bound for the coefficient $a_n$ of $F(t)$. Since the
generating function $\overline{F}(t)$ turns out to be easier to study,
in the sequel we consider only $\overline{F}(t)$. This does not give
rise to any restriction in the case of $q\geq 3$ and still provides an
upper bound for $|\mathcal{A}_n(q)|$ when $q=2$.}
\end{remark}

We note that 
Corollary~\ref{cor:exact} yields $a_n\geq \omega(\GL_n(q))/|\GL_n(q)|$
and, for $q>n$, $a_n=\omega(\GL_n(q))/|\GL_n(q)|$. So, by studying
$\overline{F}(t)$, we shall determine a good description for
$\omega(\GL_n(q))$. Namely, in Theorem~\ref{thm:closeformula} we prove
a closed simple formula for   
the function $\overline{F}(t)$  and in Theorem~\ref{thm:limit} we give an exact
formula for the limit $\lim_{n\to\infty }q^nb_n$.

Define the following two functions:
\begin{eqnarray}\label{def:F1F2}
F_1(t)&=&\prod_{d\geq 1}\exp\left(\frac{t^d}{d(1-q^{-d})q^d}\right)
,\\\nonumber
F_2(t)&=&\prod_{m\geq 2}\prod_{d\geq
  1}\exp\left(\frac{t^{dm}}{d(1-q^{-d})^2q^{2dm-d}}\right) 
.\\\nonumber
\end{eqnarray}

\begin{lemma}\label{lemma:splitF}
$\overline{F}(t)=F_1(t)F_{2}(t)$.
\end{lemma}
\begin{proof}
From the Taylor series for the exponential function $\exp(t)$ and from Definition~\ref{def:F1F2},
we have
\begin{eqnarray}\label{eq:part}
F_1(t)&=&\prod_{d\geq  1}
\left(
\sum_{r=0}^\infty\frac{t^{dr}}{(d(1-q^{-d})q^d)^rr!}\right),\\\nonumber
F_2(t)&=&\prod_{m\geq 2}
\left(
\prod_{d\geq  1}
\left(
\sum_{r=0}^\infty\frac{t^{dmr}}{(d(1-q^{-d})^2q^{2dm-d})^rr!}
\right)\right).\nonumber
\end{eqnarray}
By expanding the infinite products in Equation~\ref{eq:part} for
$F_1(t)F_2(t)$, in order to 
obtain a summand of degree $n$ we have to choose (for each $d,m\geq
1$)  a term of degree $dmr_{d,m}$ from the series 
$$
\sum_{r=0}^\infty\frac{t^{dr}}{(d(1-q^{-d})q^d)^rr!} \;(\mathrm{for }\;m=1)
\quad \mathrm{or}\quad
\sum_{r=0}^\infty\frac{t^{dmr}}{(d(1-q^{-d})^2q^{2dm-d})^rr!}
\;(\mathrm{for }\;m>1), 
$$
in such a way that $n=\sum_{d,m}dmr_{d,m}$. This yields that each
summand of degree $n$ 
obtained by expanding $F_1(t)F_2(t)$  is uniquely determined by
an element $\mu\in \Phi_n$ (by setting $\mu(d,m)=r_{d,m}$). Hence, the
coefficient of degree $n$ in 
 $F_1(t)F_2(t)$ is 

$$\sum_{\mu\in \Phi_n}\left(
\prod_{d\geq  1}\frac{1}{(d(1-q^{-d})q^d)^{\mu(d,1)}\mu(d,1)!}
\prod_{d\geq  1,m\geq 2}\frac{1}{(d(1-q^{-d})^2q^{2dm-d})^{\mu(d,m)}\mu(d,m)!}
\right).
$$
Applying Proposition~\ref{prop:numberA_n} and Equation~\ref{eq:1}, we
see that the coefficient of  $t^n$ in $\overline{F}(t)$ equals the
coefficient of $t^n$ in $F_1(t)F_2(t)$. Thus $\overline{F}(t)=F_1(t)F_2(t)$.
\end{proof}

In the rest of this section we use the theory of symmetric functions
to obtain a closed simple formula for the generating functions
$F_1(t)$ and $F_2(t)$. We start by recalling some well-known results
and definitions from~\cite{macdonald}. Let $X=\{x_i\}_{i\geq 1}$ be an
infinite set of variables and $\Lambda$  be the \emph{graded ring of symmetric
functions} on $X$
(see~\cite[Section~I.2]{macdonald}). A \emph{partition} is a
sequence $\lambda=(\lambda_1,\lambda_2,\ldots)$ of non-negative
integers in decreasing order and containing only finitely many non-zero
terms. We write $|\lambda|=\sum_{i}\lambda_i$ and 
$X^\lambda=x_1^{\lambda_1}x_2^{\lambda_2}\cdots$. Let $\lambda$ be a
partition. The function
$$m_\lambda(X)=\sum_\alpha X^\alpha$$
summed over all distinct permutations $\alpha$ of $\lambda$ is 
a symmetric function in $\Lambda$, and $m_\lambda$ is called a \emph{monomial
  symmetric function}. By definition, $m_\lambda$ is a homogeneous
function of degree $|\lambda|$. For each $d\geq 0$ the $d$th 
\emph{complete symmetric function} $h_d(X)$ is the sum of all monomial
symmetric functions of degree $d$, so that 
$$h_d(X)=\sum_{|\lambda|=d}m_\lambda(X).$$
The generating function for the complete symmetric functions is
$H_X(t)=\sum_{d\geq 0}h_d(X)t^d$ (the label $X$ in $H_X$ is needed in
order to record the set of variables $X$).  It is proven in~\cite[$(2.5)$,
  page~$14$]{macdonald} that 
\begin{equation}\label{eq:Hclose}
H_X(t)=\prod_{i\geq 1}(1-x_it)^{-1}.
\end{equation}

For each $d\geq 1$ the $d$th \emph{power sum symmetric function} is
$p_d(X)=\sum_i x_i^d$. The 
generating function for the power sum symmetric functions is defined as $P_X(t)=\sum_{d\geq
  1}p_d(X)t^{d-1}$. It is proven in~\cite[$(2.10)$,
  page~$16$]{macdonald} that 
\begin{equation}\label{eq:pfromh}
P_X(t)=\frac{d}{dt}\log H_X(t).
\end{equation}
From Equation~\ref{eq:pfromh}, we see that $k+\sum_{d\geq 1}p_d(X) t^d/d=\log
H_X(t)$, for some integer $k$. Since $H_X(0)=h_0(X)=1$, we get that $\log
H_X(0)=0$ and so $k=0$. Thence $H_X(t)=\exp(\sum_d p_d(X)t^d/d)$. This
shows that  
\begin{eqnarray}\label{eq:H}
H_X(t)&=&\prod_{d\geq 1}\exp\left(\frac{p_d(X)t^d}{d}\right).
\end{eqnarray} 

Write 
\begin{equation}\label{def:varphi}
\varphi_0(x)=1 \qquad \textrm{and} \qquad \varphi_d(x)=(1-x)\cdots (1-x^d),
\end{equation} for
$d\geq 1$.  

In the following lemma we establish two simple formulae for $F_1(t)$.
\begin{lemma}\label{lemma:F1c1}

\begin{description}
\item[$(i)$]$F_1(t)=\prod_{i\geq 0}(1-q^{-(i+1)}t)^{-1}$;
\item[$(ii)$]$F_1(t)=\sum_{d=0}^\infty
  \frac{t^d}{q^d\varphi_d(q^{-1})}$.
\end{description}
\end{lemma}
\begin{proof}
Consider the set of variables $Q=\{q^{-(i-1)}\}_{i\geq 1}$,
that is $Q$ is obtained by specialising $x_i=q^{-(i-1)}$ (for $i\geq
1$).
From~\cite[Example~$4$, page~$19$]{macdonald}, we have that the $d$th power sum
symmetric function $p_d(Q)$ on the set of variables $Q$ satisfies
$p_d(Q)=(1-q^{-d})^{-1}$.  So, 
Equation~\ref{eq:H} yields 
\begin{eqnarray*}
H_Q(q^{-1}t)&=&\prod_{d\geq
  1}\exp\left(\frac{p_d(Q)t^d}{dq^d}\right)=\prod_{d\geq
  1}\exp\left(\frac{t^d}{d(1-q^{-d})q^d}\right)=F_1(t).
\end{eqnarray*}
Equation~\ref{eq:Hclose} gives $H_Q(q^{-1}t)=\prod_{i\geq
  1}(1-q^{-(i-1)}(q^{-1}t))^{-1}=\prod_{i\geq 0}(1-q^{-(i+1)}t)^{-1}$
and $(i)$ follows.

From~\cite[Example~$4$, page~$19$]{macdonald}, we have that the $d$th
complete symmetric function $h_d(Q)$ on the set of variables $Q$
satisfies $h_d(Q)=1/\varphi_d(q^{-1})$. Thus, we obtain $H_Q(t)=\sum_{d=0}^\infty
t^d/\varphi_d(q^{-1})$ and $H_Q(q^{-1}t)=\sum_{d=0}^\infty
t^d/q^d\varphi_d(q^{-1})$, which we showed above is $F_1(t)$, and
$(ii)$ is proved.
\end{proof}

The argument for obtaining a simple formula for $F_2(t)$ is very similar
to Lemma~\ref{lemma:F1c1}, but a little trickier. We start with a
definition. For each $m\geq 2$, define
\begin{equation}\label{eq:F2m}
F_2^{(m)}(t)=\prod_{d\geq 1}\exp\left(\frac{t^{dm}}{d(1-q^{-d})^2q^{2dm-d}}\right).
\end{equation}

\begin{lemma}\label{lemma:F2c2}
\begin{description}
\item[$(i)$]$F_2^{(m)}(t)=\prod_{i,j\geq 0}(1-q^{-(i+j+2m-1)}t^m)^{-1}$;
\item[$(ii)$]$F_2(t)=\prod_{m\geq 2}\prod_{i,j\geq 0}(1-q^{-(i+j+2m-1)}t^m)^{-1}$.
\end{description}
\end{lemma}
\begin{proof}
Consider two infinite sets of variables $X=\{x_i\}_{i\geq 1}$ and
$Y=\{y_i\}_{i\geq 1}$. From $X$ and $Y$ consider the infinite set of
variables $Z=\{x_iy_j\}_{i,j\geq 1}$. By definition
of the $d$th power sum symmetric function, we have 
\begin{eqnarray}\label{eq:aux}
p_d(Z)&=&\sum_{i,j\geq 1}(x_iy_j)^d=\sum_{i\geq
  1}x_i^d\left(\sum_{j\geq 1}y_j^d\right)\\\nonumber
&=&\left(\sum_{i\geq  1}x_i^d\right)\left(\sum_{j\geq
  1}y_j^d\right)=p_d(X)p_d(Y).
\end{eqnarray}
Consider the set of variables $Q=\{q^{-(i-1)}\}_{i\geq 1}$ obtained by
specialising $x_i=q^{-(i-1)}$ (or $y_i=q^{-(i-1)}$), for $i\geq
1$. Also, consider the set of variables $Q'=\{q^{-(i+j-2)}\}_{i,j\geq
  1}$ obtained by specialising $x_iy_j=q^{-(i-1)}q^{-(j-1)}$.
Under this assignment, we obtain
from~\cite[Example~$4$, page~$19$]{macdonald} and
Equation~\ref{eq:aux} that
$p_{d}(Q')=p_d(Q)^2=(1-q^{-d})^{-2}$. 

Equation~\ref{eq:H} yields
\begin{eqnarray*}
H_{Q'}(q(q^{-2}t)^m)&=&\prod_{d\geq
  1}\exp\left(\frac{p_d(Q')(q(q^{-2}t)^m)^d}{d}\right)=
\prod_{d\geq 1}\exp\left(\frac{p_d(Q')t^{dm}}{dq^{2dm-d}}\right)\\
&=&\prod_{d\geq  1}\exp\left(\frac{t^{dm}}{d(1-q^{-d})^2q^{2dm-d}}\right)=F_2^{(m)}(t)
\end{eqnarray*}
using Equation~\ref{eq:F2m}. Finally, Equation~\ref{eq:Hclose} gives $$H_{Q'}(q(q^{-2}t)^m)=
\prod_{i,j\geq 1}(1-q^{-(i+j-2)}(q(q^{-2}t)^m))^{-1}=
\prod_{i,j\geq 1}(1-q^{-(i+j+2m-3)}t^m)^{-1}$$
and $(i)$ follows. Part~$(ii)$ follows from the definitions of $F_2$
and $F_2^{(m)}$ in Equations~\ref{def:F1F2} and~\ref{eq:F2m}.
\end{proof}

The following theorem gives a closed simple formula for $\overline{F}(t)$.

\begin{theorem}\label{thm:closeformula}
\begin{equation*}
\overline{F}(t)=\left(\prod_{i=0}^\infty(1-q^{-(i+1)}t)^{-1}\right)\left(
\prod_{m\geq 2,i,j\geq 0}(1-q^{-(i+j+2m-1)}t^m)^{-1}\right).
\end{equation*}
\end{theorem}
\begin{proof}
As $\overline{F}(t)=F_1(t)F_2(t)$, the theorem follows from 
Lemmas~\ref{lemma:F1c1},~\ref{lemma:F2c2}.
\end{proof}

By using the formula in Theorem~\ref{thm:closeformula} one can easily
obtain the first few values of $b_n$, where $\overline{F}(t)=\sum_{n\geq
  0}b_nt^n$. For instance, Table~\ref{table:1} in
Section~\ref{sec:introd} was obtained by expanding the terms in $t$ of degree
$\leq 6$ in the infinite products of $\overline{F}(t)$.

Using Lemma~\ref{lemma:F1c1}~$(ii)$, we show that
$\{q^nb_n\}_{n\geq 0}$ is an increasing sequence.

\begin{theorem}\label{thm:monotonic}
For each $n\geq 0$, $q^nb_n<q^{n+1}b_{n+1}$, where the $b_n$ are as in
Definition~\ref{def:F}. Moreover, if
$q>2$ then the sequence $\{q^n|\mathcal{A}_n(q)|/|\GL_n(q)|\}_{n\geq 0}$
is increasing.
\end{theorem}

\begin{proof}Write $F_2(t)=\sum_{d\geq 0}c_dt^d$. It is clear from
  Equation~\ref{eq:part} that $c_d\geq 0$. By Lemma~\ref{lemma:F1c1},
  we have $F_1(t)=\sum_{d\geq 0 
  }\frac{t^d}{q^d\varphi_d(q^{-1})}$ with $\varphi_d$ as in
  Equation~\ref{def:varphi}. Thus it follows from
  Lemma~\ref{lemma:splitF} that
  $b_n=\sum_{d=0}^n\frac{c_{n-d}}{q^d\varphi_d(q^{-1})}$. Now
  $\varphi_{d+1}(q^{-1})=\varphi_d(q^{-1})(1-q^{-(d+1)})<\varphi_d(q^{-1})$
  and $c_d\geq 0$, and hence

\begin{eqnarray*}
q^nb_n&=&q^n\sum_{d=0}^n\frac{1}{q^d\varphi_d(q^{-1})}c_{n-d}=
\sum_{d=0}^n\frac{1}{\varphi_{d}(q^{-1})}(q^{n-d}c_{n-d})\\
&< &\sum_{d=0}^n\frac{1}{\varphi_{d+1}(q^{-1})}(q^{n-d}c_{n-d})=q^{n+1}\sum_{d=0}^{n}\frac{1}{q^{d+1}\varphi_{d+1}(q^{-1})}c_{n-d}\\ 
&=&q^{n+1}\sum_{d=1}^{n+1}\frac{1}{q^d\varphi_d(q^{-1})}c_{n+1-d}\leq
q^{n+1}\sum_{d=0}^{n+1}\frac{1}{\varphi_d(q^{-1})}c_{n+1-d}=q^{n+1}b_{n+1}.   
\end{eqnarray*}

If $q >2$, then $b_n=|\mathcal{A}_n(q)|/|\GL_n(q)|$ by
Remark~\ref{rk:6.2} and the definition of $a_n$, and  the
last assertion follows. 
\end{proof}

\section{Analytic properties of the generating function
  $\overline{F}(t)$}\label{sec:ana} 
Before studying analytically the functions
$\overline{F}(t),F_1(t),F_2(t)$ we have 
to collect some numerical information that will be used later.

\begin{lemma}\label{lemma:estimates}Set
  $l(q)=\prod_{k=1}^\infty(1-q^{-k})^{-k(k+1)/2-1}$. We have:
\begin{description}
\item[$(a)$]$l(q)>1+2q^{-1}+7q^{-2}+19q^{-3}$;
\item[$(b)$]$l(q)<(1-q^{-1}-q^{-2})^{-1}\exp(q^{-1}/(1-q^{-1})^3)
\exp(q^{-2}(1+q^{-1})/2(1-q^{-2})^4)$; 
\item[$(c)$]for $q>2$ , $l(q)<1+2q^{-1}+7q^{-2}+114q^{-3}$; 
\item[$(d)$]if $q=2$, then $395.0005>l(2)>278.98$.
\end{description}
\end{lemma}
\begin{proof}Part~$(a)$ follows by expanding in powers of $q$ the
  first three terms 
  $(1-q^{-1})^{-2}$, $(1-q^{-2})^{-4}$ and $(1-q^{-3})^{-7}$(for $k=1,2,3$) of the
  infinite product $l(q)$ and noticing that $(1-q^{-k})^{-1}>1$, for $k\geq
  1$. 

Next, we consider an 
  upper bound for $l(q)$. First, we recall that by the Binomial Theorem,
  $(1-x)^{-s}=\sum_{k= 0}^\infty{k+s-1\choose s-1}x^k$. In particular,
\begin{equation}\label{eq:auxx}
\sum_{k=1}^\infty {k+1\choose 2}x^k=
 x\sum_{k=1}^\infty{k+1\choose 2}x^{k-1}=
x\sum_{k'=0}^\infty{k'+2\choose 2}x^{k'}=\frac{x}{(1-x)^{3}}.
\end{equation}
Set $L(q)=\prod_{k=1}^\infty(1-q^{-k})^{-k(k+1)/2}$. We have

\begin{eqnarray*}
\log(L(q))&=&-\sum_{k=1}^{\infty}\frac{k(k+1)}{2}\log(1-q^{-k})
=\sum_{k=1}^\infty\frac{k(k+1)}{2}\left(\sum_{m=1}^\infty\frac{q^{-km}}{m}\right)\\
&=&\sum_{m=1}^\infty\frac{1}{m}\left(\sum_{k=1}^\infty{k+1\choose 2}q^{-km}\right)
=\sum_{m=1}^\infty\frac{q^{-m}}{m(1-q^{-m})^3}\qquad(\dag) \\
&<&\frac{q^{-1}}{(1-q^{-1})^3}+\frac{1}{2}\sum_{m=2}^{\infty}\frac{q^{-m}}{(1-q^{-m})^3}\\
&<&\frac{q^{-1}}{(1-q^{-1})^3}+\frac{1}{2(1-q^{-2})^3}\sum_{m=2}^{\infty}q^{-m}\\
&=&\frac{q^{-1}}{(1-q^{-1})^3}+\frac{q^{-2}}{2(1-q^{-2})^3(1-q^{-1})},\\
\end{eqnarray*}
where in $(\dag)$ we used Equation~\ref{eq:auxx}.
From~\cite[Lemma~$3.5$]{NP}, we have
$\prod_{k=1}^\infty(1-q^{-k})^{-1}<(1-q^{-1}-q^{-2})^{-1}$. Therefore,
Part~$(b)$ follows.

Now, assume $q>2$ and set $T=\exp(q^{-1}/(1-q^{-1})^3)$. So, 

\begin{eqnarray*}
T&=&\sum_{r=0}^\infty\frac{(q^{-1}/(1-q^{-1})^3)^r}{r!}<
\sum_{r=0}^3\frac{(q^{-1}/(1-q^{-1})^3)^r}{r!}+
\sum_{r=4}^\infty \left(\frac{q^{-1}/(1-q^{-1})^3}{2}\right)^r\\
&=&\sum_{r=0}^3\frac{(q^{-1}/(1-q^{-1})^3)^r}{r!}+
\frac{\left(\frac{q^{-1}/(1-q^{-1})^3}{2}\right)^4}{1-\frac{q^{-1}/(1-q^{-1})^3}{2}}<1+q^{-1}+\frac{7}{2}q^{-2}+41q^{-3},  
\end{eqnarray*}
where the first inequality uses $r!\geq 2^{r}$ (for $r\geq 4$) and
the last inequality is obtained by 
expanding in powers of $q$ and using the fact that
$q>2$. 

With similar
computations we get $\exp(q^{-2}(1+q^{-1})/2(1-q^{-2})^4)<1+q^{-2}/2+2q^{-3}$ and 
$(1-q^{-1}-q^{-2})^{-1}<1+q^{-1}+2q^{-2}+7q^{-3}$. Now, from
Part~$(b)$ we have 
\begin{eqnarray*}
l(q)&<&\left(1+q^{-1}+2q^{-2}+7q^{-3}\right)
\left(1+q^{-1}+\frac{7}{2}q^{-2}+41q^{-3}\right)
\left(1+\frac{1}{2}q^{-2}+2q^{-3}\right)\\
&<&1+2q^{-1}+7q^{-2}+114q^{-3}
\end{eqnarray*} and
Part~$(c)$ follows. 

The lower bound in Part~$(d)$ is obtained by
computing $\prod_{k=1}^{30}(1-q^{-k})^{-k(k+1)/2-1}$ with
  $q=2$ and the upper bound is obtained by substituting $q=2$ in Part~$(b)$.
\end{proof}

In the following two propositions, we study some analytic properties of
$\overline{F}(t),F_1(t),F_2(t)$. 
\begin{proposition}\label{prop:analyticF1}
$F_1(t)$ is analytic on a disk of radius $q$.
Also, $F_1(t)$ has a simple pole at $t=q$  and $(1-q^{-1}t)F_1(t)$ is
analytic on a disk of radius $q^2$.
\end{proposition}

\begin{proof}
From Definition~\ref{def:F1F2}, we obtain
\begin{equation}\label{eq:2}
F_1(t)=
\prod_{d=1}^{\infty}
\exp\left(
\frac{(q^{-1}t)^d}{d(1-q^{-d})}
\right)=\exp\left(
\sum_{d=1}^\infty\frac{(q^{-1}t)^d}{d(1-q^{-d})}
\right).\\
\end{equation}
Next, we determine where the series in Equation~\ref{eq:2} is
absolutely convergent. We get

\begin{equation}\label{eq:3}
\sum_{d=1}^\infty\frac{|q^{-1}t|^d}{d(1-q^{-d})}\leq
\sum_{d=1}^\infty \frac{|q^{-1}t|^d}{1-q^{-d}}\leq \frac{1}{1-q^{-1}}\sum_{d=1}^\infty|q^{-1}t|=\frac{|q^{-1}t|}{(1-q^{-1})(1-|q^{-1}t|)}. 
\end{equation} 
Since $1/(1-q^{-1}t)$ is analytic on a disk of radius $q$ and has a
simple pole in $t=q$, the equivalent result for $F_1(t)$ follows at
once.

It remains to show
that $(1-q^{-1}t)F_1(t)$ is analytic on  a disk of radius $q^2$. Since
$\varphi_r(q^{-1})=\varphi_{r-1}(q^{-1})(1-q^{-r})$ (for $r\geq 1$), from
Lemma~\ref{lemma:F1c1}~$(ii)$ we get  
\begin{eqnarray*}
(1-q^{-1}t)F_1(t)&=&(1-q^{-1}t)\sum_{r=0}^\infty\frac{t^r}{q^r\varphi_r(q^{-1})}=
1+\sum_{r=1}^\infty\left(\frac{t^r}{q^r\varphi_r(q^{-1})}-\frac{t^r}{q^r\varphi_{r-1}(q^{-1})}\right)\\
&=&1+\sum_{r=1}^\infty\frac{q^{-r}t^r}{q^r\varphi_r(q^{-1})}=\sum_{r=0}^\infty\frac{t^r}{q^{2r}\varphi_r(q^{-1})}.  
\end{eqnarray*}
As, $$\sum_{r=0}^\infty\frac{|t|^r}{q^{2r}\varphi_r(q^{-1})}\leq
\prod_{i=1}^\infty(1-q^{-i})^{-1}\sum_{r=0}^\infty
|q^{-2}t|^r=\prod_{i=1}^\infty(1-q^{-i})^{-1}\frac{1}{1-|q^{-2}t|},$$
we get that $(1-q^{-1}t)F_1(t)$ is analytic on a disk of radius $q^2$.
\end{proof}

\begin{proposition}\label{prop:analyticF2}
$F_2(t)$ is analytic on a disk of radius $q^{3/2}$.
\end{proposition}
\begin{proof}
We argue as in Proposition~\ref{prop:analyticF1}. We have
\begin{eqnarray*}
F_2(t)&=&
\prod_{m=2}^\infty
\prod_{d=1}^\infty
\exp\left(\frac{t^{dm}}{d(1-q^{-d})^2q^{2dm-d}}\right)=\exp\left(\sum_{m=2}^\infty\sum_{d=1}^\infty\frac{t^{dm}}{d(1-q^{-d})^2q^{2dm-d}}\right)\\
&=&\exp\left(
\sum_{d=1}^\infty
\frac{t^{2d}q^d}{d(1-q^{-d})^2q^{4d}}
\sum_{m=0}^\infty\frac{t^{dm}}{q^{2dm}}
\right)
=
\exp\left(\sum_{d=1}^\infty\frac{(q^{-2}t)^{2d}q^d}{d(1-q^{-d})^2(1-(q^{-2}t)^{d})}\right).
\end{eqnarray*}

Set $z=q^{-2}t$. We obtain

\begin{equation}\label{eq:5}
\log(F_2(t))=\sum_{d=1}^\infty \frac{z^{2d}q^d}{d(1-q^{-d})^2(1-z^{d})}.
\end{equation}
We prove that for $|z|<q^{-1/2}$ (i.e. $|t|<q^{3/2}$) the series in
Equation~\ref{eq:5} is absolutely convergent. Note that 

$$\left|\frac{1}{1-z^d}\right|\leq \frac{1}{1-q^{-d/2}}\leq
\frac{1}{1-q^{-1/2}}\qquad \textrm{and}\qquad
\frac{1}{d(1-q^{-d})^2}\leq \frac{1}{(1-q^{-1})^2}. $$

Thus the series in Equation~\ref{eq:5} is absolutely convergent if
$$\sum_{d=1}^\infty |z^2q|^d=\frac{1}{1-|z^2q|}$$ is
absolutely convergent. As $|z|<q^{-1/2}$, we get $|z^2q|<1$, so
$F_2(t)$ is analytic for $|t|<q^{3/2}$.  
\end{proof}

We finally determine the asymptotic behaviour of $\{b_n\}_{n\geq 1}$
(and so also for $\{a_n\}_{n\geq 1}$ when $q>2$).
\begin{theorem}\label{thm:limit}We have $$l(q)=\lim_{n\to\infty}q^nb_n=
\prod_{k=1}^\infty(1-q^{-k})^{-k(k+1)/2-1}$$
and also $|q^nb_n-l(q)|=o(r^{-n/2})$, for every $0<r<q$.
\end{theorem}
\begin{proof}
From Propositions~\ref{prop:analyticF1},~\ref{prop:analyticF2}, we get
that  $\overline{F}(t)$ is an analytic function on a disk of radius $q$, the
point $t=q$ is 
a simple pole for $\overline{F}(t)$ and
$f(t)=(1-q^{-1}t)\overline{F}(t)$ is an analytic 
function on disk of
radius $q^{3/2}$.  In particular, by~\cite[Lemma~$1.3.3$]{FNP}, we
get that $\lim_{n\to\infty}q^nb_n=f(q)$ and
$|b_n-f(q)/q^n|=o(r^{-3n/2})$, for every $0<r<q$. In particular, it
remains to compute 
$f(q)$.

From Lemma~\ref{lemma:F2c2}~$(ii)$, we get
$F_2(t)=\prod_{m\geq 2,i,j\geq 0}(1-q^{-(i+j+2m-1)}t^m)^{-1}$. In particular,
$F_2(q)=\prod_{m\geq 2,i,j\geq
  0}(1-q^{-(i+j+m-1)})^{-1}$. Now, given $k\geq 1$, there exist
$k(k+1)/2$ choices of $(i,j,m)$ such that $k=i+j+m-1$, with $m\geq
2$ and $i,j\geq 0$. Therefore,
$F_2(q)=\prod_{k=1}^\infty(1-q^{-k})^{-k(k+1)/2}$.

From Lemma~\ref{lemma:F1c1}, we get
$(1-q^{-1}t)F_1(t)=\prod_{i=1}^\infty(1-q^{-(i+1)}t)^{-1}$. So,
$f(q)=\prod_{i=1}^\infty(1-q^{-i})^{-1}\prod_{k=1}^\infty(1-q^{-k})^{-k(k+1)/2}$
and the theorem follows. 
\end{proof}

Summing up, Theorem~\ref{thm:limiting} follows from
Theorems~\ref{thm:monotonic} and~\ref{thm:limit}. Therefore, all the
results mentioned in the introduction are now proved.

\begin{remark}\label{rm:1}{\rm Since $|\GL_n
  (q)|=q^{n^2}\varphi_n(q^{-1})$,
  Theorem~\ref{thm:limit} yields that 
  $b_n|\GL_n(q)|$ is a polynomial $p(q)$ in $q$ of
  degree $q^{n^2-n}$. So, write
  $p(q)=\sum_{i=0}^{n^2-n}\alpha_iq^{n^2-n-i}$ and pick $r$ such that
  $0<r<q$. 

As $|q^nb_n-l(q)|=o(r^{-n/2})$, we obtain that 
$|p(q)-q^{n^{2}-n}\varphi_n(q^{-1})l(q)|=o(r^{n^2-n-n/2})=o(r^{n^2-3n/2})$. 
Furthermore, since 
$\varphi_n(q^{-1})l(q)=\prod_{k=n+1}^\infty(1-q^{-k})^{-1}
\prod_{k=1}^\infty(1-q^{-k})^{-k(k+1)/2}$,
we see
that $$\left|\sum_{i=0}^{n^2-n}\alpha_iq^{n^2-n-i}
-q^{n^2-n}\prod_{k=1}^{\infty}(1-q^{-k})^{-k(k+1)/2}\right|=o(r^{n^2-3n/2}).$$   
It follows that the first $\lfloor n/2\rfloor$ coefficients
$\alpha_0,\ldots,\alpha_{\lfloor n/2\rfloor-1}$ are obtained 
from the series expansion of
$\prod_{k=1}^{\infty}(1-q^{-k})^{-k(k+1)/2}$. Such a sequence is
described and  studied in~\cite{sloane}.}
\end{remark}

\begin{remark}\label{rm:2}{\rm 
It was conjectured in~\cite{AP} that, for
$q>n$, $$\omega(\GL_n(q))\geq
q^{n^2-n}+\frac{|\GL_n(q)|}{q(q-1)^n}+\frac{|\GL_n(q)|}{q^{(n^2-n)/2}(q-1)^2}.$$ 
The second summand on the right hand side of this inequality is
asymptotic to
$q^{n^2-n-1}+(n-1)q^{n^2-n-2}+O(q^{n^2-n-3})$. Therefore,
Remark~\ref{rm:1} and 
Table~\ref{table:1} yields that this conjecture is 
incorrect for $n\geq 6$.}
\end{remark}

\section{Concluding comments}
We conclude by noticing that we can exploit the  theory presented in
this paper  further in order to obtain a refinement of
Corollary~\ref{cor:exact} in the case that $q\leq 
n$. We give an example in the case  $n=q$ and $q>2$.

From the proof of Theorem~\ref{thm:exact} we have that  if $n=q$, then
all $A_\mu$ (for $\mu\in \Phi_n$) are centralisers of 
cyclic matrices except if $\sum_{m}\mu(1,m)>q-1$. Since
$q=\sum_{d,m}\mu(d,m)dm$, this implies $\sum_{m}\mu(1,m)=q$ and so
$\mu=\mu_0$ where

\[
\mu_0(d,m)=\left\{
\begin{array}{ccl}
0&&\textrm{if }d\geq 2 \textrm{ or }m\geq 2,\\
q&&\textrm{if }d=1 \textrm{ and }m=1.\\
\end{array}
\right.
\]

By Definition~\ref{def:Amu}, $A_{\mu_0}$ is the group of
diagonal matrices, that is, the split torus
of size $(q-1)^q$. By Proposition~\ref{prop:numberA_n} we have
$|N_{\GL_q(q)}(A_{\mu_0})|=(q-1)^qq!$. 

Since there are only $q-1$ distinct
eigenvalues 
available,
any $g\in \GL_q(q)$ giving a decomposition of $V_g$ as
$\oplus_{i=1}^qV_i$ into $1$-dimensional spaces has an eigenvalue with
multiplicity $\geq 2$. Therefore, we see that $g$ is centralised by some cyclic
matrix $g_\mu \in A_\mu$, for $\mu\neq \mu_0$. In particular,
following the proof of Corollary~\ref{cor:exact} we have that 

$$\omega(\GL_q(q))=|\mathcal{A}_q|-\frac{|\GL_q(q)|}{(q-1)^qq!}.$$

It is not clear to the authors of this paper whether there is a general
theory for small $q$ and large $n$ with a tractable formula for
$\omega(\GL_n(q))$. 

\bibliographystyle{amsplain}

\begin{thebibliography}{10}

\bibitem{AAM}
{\sc A. Abdollahi, A. Akbari, H. R. Maimani}, Non-commuting graph of a
group, {\it J. Algebra} {\bf 298} (2006), 468--492.

\bibitem{AP}
{\sc A. Azad, C. E. Praeger}, Maximal sets of pairwise
noncommuting elements of finite three-dimensional general linear
groups, {\it Bull. Austral. Math. Soc.} {\bf 80} (2009), 91--104.

\bibitem{Isaacs}
{\sc E. A. Bertram}, Some applications of graph theory to
finite groups, {\it Discrete Math.} {\bf 44} (1983), 31--43.

\bibitem{Brown}
{\sc R. Brown}, Minimal covers of $S_n$ by abelian subgroups and
maximal subsets of pairwise noncommuting elements, {\it
  J. Combin. Theory Ser. A} {\bf 49} No. 2 (1988), 294-307. 

\bibitem{Carter}
{\sc R. W. Carter}, {\it Finite groups of {L}ie type, Conjugacy
  Classes and Complex Characters}, Wiley,
Chichester, 1993.

\bibitem{Chin}
{\sc A. M. Y. Chin}, On non-commuting sets in an extra
special $p$-group, {\it J. Group Theory} {\bf 8} (2005), 189--194. 

\bibitem{CuRe}{\sc C. W. Curtis, I. Reiner}, Representation theory of
  finite groups and associative algebras, Wiley, Pure and Applied
  Mathematics Volume XI, (1962). 

\bibitem{Dixon}{\sc J. D. Dixon}, Maximal abelian subgroups of the symmetric
  groups, {\it Can. J. Math. } {\bf 23} (1971), 426--438.

\bibitem{Fu}{\sc J. Fulman}, Cycle indices for the finite classical
  groups, {\it J. Group Theory} {\bf 2} (1999), 251--289.

\bibitem{FNP}{\sc J. Fulman, P. N. Neuman, C. E. Praeger}, A Generating
  Function Approach to the Enumeration of Matrices in Classical Groups
  over Finite Fields, {\it Mem.  Amer. Math.
    Soc.} {\bf 176}, (2005). 

\bibitem{Humphreys}{\sc J. E. Humphreys}, Conjugacy classes in semisimple
  algebraic groups, {\it Mathematical Surveys and Monographs} {\bf 43}, 1995. 

\bibitem{macdonald}{\sc I. G. Macdonald}, Symmetric functions and Hall
  polynomials, Claredon Press, Oxford, (1979).

\bibitem{Mckenzie}
{\sc V. Faber, R. Laver and R. Mckenzie}, Covering of groups by
abelian subgroups, {\it Can. J. Math.} {\bf 30}, (1978), 933--945.

\bibitem{Neumann}
{\sc B. H. Neumann}, A problem of Paul Erd\H{o}s on groups,
{\it J. Aust. Math. Soc. Ser. A} {\bf 21} (1976), 467--472.

\bibitem{NP}
{\sc  P. M. Neumann , C. E. Praeger}, Cyclic matrices over
finite fields, {\it J. London Math. Soc.} {\bf 52} (1995), 263--284.

\bibitem{NPMeataxe}{\sc P. M. Neumann, C. E. Praeger},
Cyclic matrices and the MEATAXE,
{\it Ohio State Univ. Math. Res. Inst. Publ.} {\bf 8}, de Gruyter, Berlin, 2001, 291--300. 

\bibitem{Pyber}
{\sc L. Pyber}, The number of pairwise non-commuting
elements and the index of the centre in a finite group,  {\it J. London
Math. Soc.} (2) {\bf 35} (1987), 287--295.

\bibitem{sloane}{\sc N. J. A. Sloane}, Ed. (2008), The On-Line Encyclopedia
  of Integer Sequences, 
published electronically at
\texttt{www.research.att.com/~njas/sequences/Sequence A000294}.

\bibitem{T}
{\sc H. J. Tomkinson}, Groups covered by finitely many cosets or subgroups,
{\it Comm. in Algebra} {\bf 15} (1987), 845--859.

\bibitem{Vdovin}{\sc E. P. Vdovin}, The number of subgroups with trivial
  unipotent radicals in finite groups of Lie type, {\it J. Group
    Theory} {\bf 7} (2004), 99--112.

\bibitem{Wa}{\sc G. E. Wall}, Counting cyclic and separable matrices
  over a finite field, {\it Bulletin Australian Math. Soc.} {\bf 60}
  (1999), 253--284.
\end{thebibliography}

\end{document}